\documentclass[11pt,onecolumn,draftcls]{IEEEtran}
\usepackage{verbatim}
\usepackage{amsfonts,amsmath,mathrsfs,amssymb,amsbsy}
\usepackage[final]{graphicx}
\usepackage{times,cite}
\usepackage{subfig}
\newtheorem{theorem}{Theorem}[section]
\newtheorem{corrly}{Corollary}

\newtheorem{lemma}{Lemma}
\newtheorem{definition}{Definition}

\newtheorem{algorithm}{Algorithm}

\newcommand{\Prob}{\mathbb{P}}
\newcommand{\Expect}{\mathbb{E}}
\newcommand{\indic}{\mathbb{I}}

\newcommand{\epp}{\epsilon}

\usepackage{color,soul}
\definecolor{lightblue}{rgb}{.7, .8, 1}
\definecolor{lightgreen}{rgb}{.6, 1, .6}
\usepackage{color}
\definecolor{brown}{rgb}{1,0.38,0.03}

\definecolor{OliveGreen}{rgb}{.2,0.6,0.2}
\definecolor{BrickRed}{rgb}{.7,0.2,0.2}

\newcommand{\ignore}[1]{} 
\newcommand{\ADD}{{\mathsf{ADD}}}
\newcommand{\CADD}{{\mathsf{CADD}}}

\newcommand{\WADD}{{\mathsf{WADD}}}
\newcommand{\FAR}{{\mathsf{FAR}}}

\newcommand{\PFA}{\mathsf{PFA}}

\newcommand{\ANO}{{\mathsf{ANO}}}

\newcommand{\mc}[1]{\mathcal{#1}} 

\newcommand{\Cc}{{\mc{C}}}
\newcommand{\Sc}{{\mc{S}}}
\newcommand{\Dd}{{\mc{D}}}

\newcommand{\taus}{\tau_{\scriptscriptstyle \text{S}}}
\newcommand{\tauc}{\tau_{\scriptscriptstyle \text{C}}}
\newcommand{\tausr}{\tau_{\scriptscriptstyle \text{SR}}}
\newcommand{\tausrp}{\tau_{\scriptscriptstyle \text{SRP}}}
\newcommand{\tausrr}{\tau_{\scriptscriptstyle \text{SR}-r}}
\newcommand{\taug}{\tau_{\scriptscriptstyle \text{G}}}
\newcommand{\taum}{\tau_{\scriptscriptstyle \text{M}}}
\newcommand{\htaug}{\hat{\tau}_{\scriptscriptstyle \text{G}}}
\newcommand{\htaum}{\hat{\tau}_{\scriptscriptstyle \text{M}}}
\begin{document}

\title{Quickest Change Detection}
\author{Venugopal V. Veeravalli and Taposh Banerjee\\ ECE Department and Coordinated Science Laboratory\\
1308 West Main Street, Urbana, IL~~61801\\Email: {vvv, banerje5}$@$illinois.edu. Tel: +1(217) 333-0144, Fax: +1(217) 244-1642.}
\maketitle

\centerline{\bf Abstract}
 The problem of detecting changes in the statistical properties
of a stochastic system and time series arises in various branches of science and engineering. It has a wide spectrum of important applications ranging from machine monitoring to biomedical signal processing. In all of these applications the observations being monitored undergo a change in distribution in response to a change or anomaly in the environment, and  the goal is to detect the change as quickly as possibly, subject to false alarm constraints.  In this chapter, two formulations of the quickest change detection problem, Bayesian and minimax, are introduced, and optimal or asymptotically optimal solutions  to these formulations are discussed. Then some generalizations and extensions of the quickest change detection problem are described. The chapter is concluded with a discussion of applications and open issues.


\medskip

%
%
%


\section{Introduction}
\label{sec:Intro}
The problem of quickest change detection comprises three entities: a stochastic process under observation, a change point at which the statistical properties of the process undergo a change, and a decision maker that observes the stochastic process and aims to detect this change in the statistical properties of the process. A {\em false alarm} event  happens when the change is declared by the decision maker before the change actually occurs. The general objective of the theory of quickest change detection is to design algorithms that can be used to detect the change as soon as possible,  subject to false alarm constraints.

The quickest change detection problem has  a wide range  of important applications, including biomedical signal and image processing, quality control engineering, financial markets, link failure detection in communication networks, intrusion detection in computer networks and security systems, chemical or biological warfare agent detection systems (as a protection tool against terrorist attacks), detection of the onset of an epidemic, failure
detection in manufacturing systems and large machines, target detection in surveillance systems, econometrics, seismology, navigation, speech segmentation, and the analysis of historical texts. See Section~\ref{sec:QCDapplications} for a more detailed discussion of the applications and related references.

To motivate the need for quickest change detection algorithms, in Fig.~\ref{fig:Samples_Gaussian} we plot
a sample path of a stochastic sequence whose samples are distributed as ${\cal N}(0,1)$ before the change, and distributed as ${\cal N}(0.1,1)$ after the change. For illustration, we choose time slot 500 as the change point.
As is evident from the figure, the change cannot be detected through manual inspection.
In Fig.~\ref{fig:PkEvolution_Gaussian},
we plot the evolution of the \textit{Shiryaev statistic} (discussed in detail in Section~\ref{sec:BayesCent}), computed using the samples of Fig.~\ref{fig:Samples_Gaussian}.
As seen in Fig.~\ref{fig:PkEvolution_Gaussian}, the value of the Shiryaev statistic stays
close to zero before the change point, and grows up to one after the change point.
The change is detected by using a threshold of 0.8.
\begin{figure}[!tb]
  \centering
  \subfloat[Stochastic sequence with samples from $f_0 \sim {\cal N}(0,1)$ before the change (time slot 500), and with samples from $f_1 \sim {\cal N}(0.1,1)$ after the change.]{\label{fig:Samples_Gaussian}\includegraphics[width=0.45\textwidth]{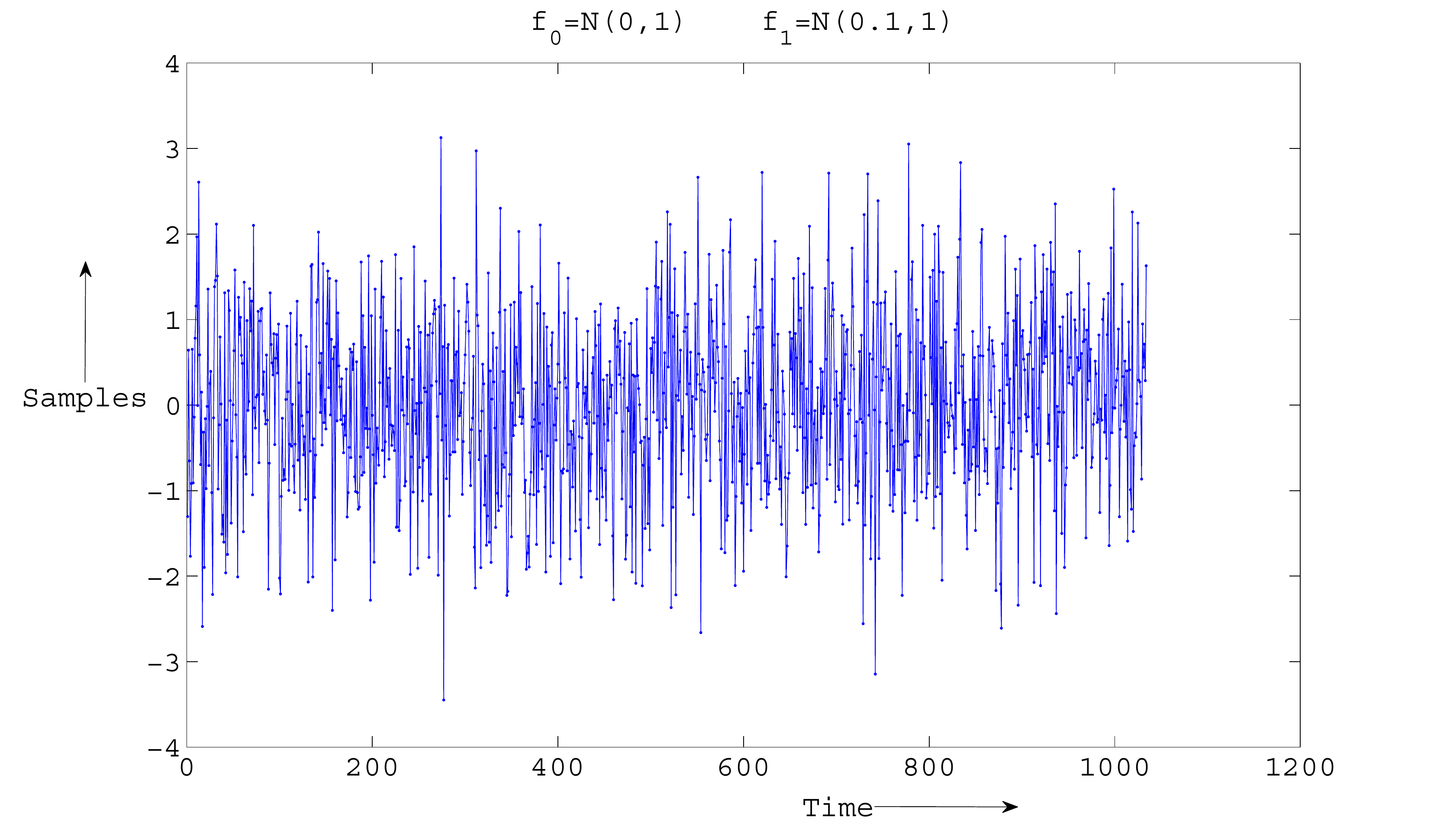}} \hspace{0.5cm}
  \subfloat[Evolution of the classical Shiryaev algorithm when applied to the samples given on the left. We see that the change is detected approximately at time slot 1000.]{\label{fig:PkEvolution_Gaussian}\includegraphics[width=0.45\textwidth]{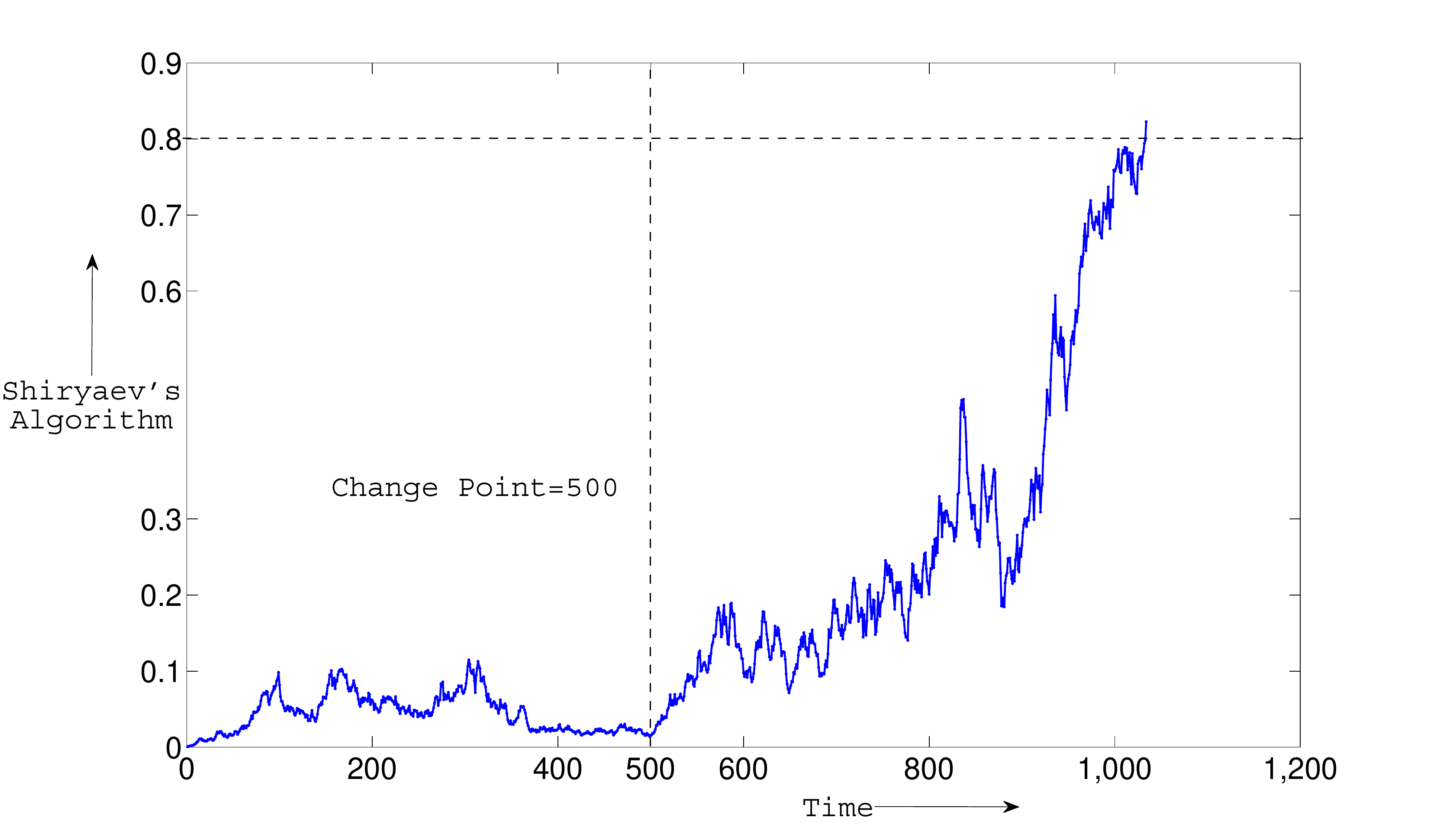}}
    \caption{Detecting a change in the mean of a Gaussian random sequence.\label{fig:CostCurvesAndEvolutionforShiryaev}
}
 \end{figure}

We also see from Fig.~\ref{fig:PkEvolution_Gaussian} that it takes  around 500 samples to detect the change after it occurs. Can we do better than that, at least on  an average? Clearly, declaring change before the change point (time slot 500) will result in zero delay,  but it will cause a false alarm. The theory of quickest change detection deals with finding algorithms that have provable optimality properties, in the sense of minimizing the average detection delay under a false alarm constraint. We will show later that the Shiryaev algorithm, employed in Fig.~\ref{fig:PkEvolution_Gaussian}, is  optimal for a certain \emph{Bayesian} model.

Earliest results on quickest change detection date back to the work of Shewhart \cite{shew-jamstaa-1925,shew-book-1931} and Page \cite{page-biometrica-1954} in the context of statistical process/quality control.
Here the {\em state} of the system is monitored by taking a sequence of measurements, and an alarm has to be raised if the measurements indicate a fault in the process under observation or if the state is \emph{out of control}. Shewhart proposed the use of a control chart to detect a change, in which  the measurements taken over time are plotted on a chart and an alarm is raised the first time the measurements fall outside some pre-specified control limits. In the Shewhart control chart procedure, the statistic computed at any given time is a function of only the measurements at that time, and not of the measurements taken in the past. This simplifies the algorithm but may result in a loss in performance (unacceptable delays when in detecting small changes). In \cite{page-biometrica-1954}, Page proposed that instead of ignoring the past observations, a weighted sum (moving average chart) or a cumulative sum (CuSum) of the past statistics  (likelihood ratios) can be used in the control chart to detect the change more efficiently. It is to be noted that the motivation in the work of Shewhart and Page was to design easily implementable schemes with good performance, rather than to design schemes that could be theoretically proven to be optimal with respect to a suitably chosen performance criterion.

Initial theoretical formulations of the quickest change detection problem were for an observation model in which, conditioned on the change point, the observations are independent and identically distributed (i.i.d.)  with some known distribution before the change point, and i.i.d. with some other known distribution after the change point. This observation model will be referred to as the {\em i.i.d. case} or {\em i.i.d model} in this  article.

The i.i.d. model was  studied by Shiryaev  \cite{shir-siamtpa-1963, shir-opt-stop-book-1978}, under the further assumption that the change point is a random variable with a known geometric distribution.  Shiryaev obtained an optimal algorithm that minimizes the {\em average detection delay} over all stopping times that meet a given constraint on the {\em probability of false alarm}. We refer to Shiryaev's formulation as the {\em Bayesian} formulation; details are provided in Section~\ref{sec:BayesCent}.

When the change point is modeled as non-random but unknown, the probability of false alarm is not well defined and therefore false alarms are quantified through the {\em mean time to false alarm} when the system is operating under the pre-change state, or through its reciprocal, which is called the {\em false alarm rate}. Furthermore, it is generally not possible to obtain an algorithm that is uniformly efficient over all possible values of the change point, and therefore a  {\em minimax} approach is required. The first minimax theory is due to Lorden \cite{lord-amstat-1971} in which he proposed a measure of detection delay obtained by taking the supremum (over all possible change points) of a worst-case delay over all possible realizations of the observations, conditioned on the change point.  Lorden showed that the CuSum algorithm of \cite{page-biometrica-1954} is asymptotically optimal according to his minimax criterion for delay, as the mean time to false alarm goes to infinity (false alarm rate goes to zero). This result was improved upon by Moustakides  \cite{mous-astat-1986} who showed that the CuSum algorithm is exactly optimal under Lorden's criterion. An alternative proof of the
optimality of the CuSum procedure is provided in \cite{ritov-astat-1990}.  See  Section~\ref{sec:MinimaxCent} for details.

Pollak \cite{poll-astat-1985}  suggested modifying Lorden's minimax criterion by replacing the double maximization of Lorden by a single maximization over all possible change points of the detection delay conditioned on the change point. He showed that an algorithm called the Shiryaev-Roberts algorithm, one that is obtained by taking a limit on Shiryaev's Bayesian solution as the geometric parameter of the change point goes to zero, is asymptotically optimal as the false alarm rate goes to zero.
It was later shown in \cite{lai-ieeetit-1998} that even the CuSum algorithm is asymptotically optimum under the Pollak's criterion, as the false alarm rate goes to zero.
Recently a family of algorithms based on the Shiryaev-Roberts statistic was shown to have strong optimality properties as the false alarm rate goes to zero. See \cite{polu-tart-mcap-2011} and Section~\ref{sec:MinimaxCent} for details.

For the case where the pre- and post-change  observations are not independent conditioned on the change point, the quickest change detection problem was studied in
the minimax setting by \cite{lai-ieeetit-1998} and in the Bayesian setting by \cite{tart-veer-siamtpa-2005}. In both of these works, an asymptotic lower bound on the delay is obtained
for any stopping rule that meets a given false alarm constraint (on false alarm rate in \cite{lai-ieeetit-1998} and on the probability of false
alarm in \cite{tart-veer-siamtpa-2005}), and an algorithm is proposed that
meets the lower bound on the detection delay asymptotically. Details are given in Section~\ref{sec:NONIIDBayes} and Section~\ref{sec:LAI}

To summarize, in Sections~\ref{sec:BayesCent} and \ref{sec:MinimaxCent}, we discuss the Bayesian and Minimax versions of the quickest change detection problem, where the  change has to be detected in a single sequence of random variables,  and where the pre- and post-change distributions are given. In Section~\ref{sec:QCDVariants}, we discuss variants and generalizations of the classical quickest change detection problem, for which significant progress has been made.  We consider the cases where the pre- or post-change distributions are not completely specified (Section~\ref{sec:GLRT}),  where there is an additional  constraint on the cost of observations used in the detection process (Section~\ref{sec:DE-QCD}), and where the change has to  detected using multiple geographically distributed sensor nodes (Section~\ref{sec:dss}).
 In Section~\ref{sec:QCDapplications} we provide a brief overview of the applications of quickest change detection.
We conclude in Section~\ref{sec:ConclusionFutureDirect} with a discussion of other possible extensions and future research directions.

For a more detailed treatment of some of the topics discussed in this chapter, we refer the reader to the books  by Poor and Hadjiliadis \cite{poor-hadj-qcd-book-2009} and Chow, Robbins and Siegmund  \cite{chow-robb-sieg-book-1971}, and the upcoming book by Tartakovsky, Nikiforov, and Basseville \cite{tart-niki-bass-2013}.
We will restrict our attention in this chapter to detecting changes in discrete-time stochastic systems; the continuous time setting is discussed in  \cite{poor-hadj-qcd-book-2009}.

In Table~\ref{Tab:Glossary}, a glossary of important symbols used in this chapter is provided.

\begin{table} [htbp]
 \centering
 \centering
 {\small
 \caption{Glossary}
\label{Tab:Glossary}
\begin{tabular}{|l|l|}
\hline
\cline{1-2}
\textbf{Symbol}&\textbf{Definition/Interpretation}\\
\hline
$o(1)$&$x=o(1)$ as $c \to c_0$, if $\forall \epsilon>0$, $\exists \delta>0$ s.t., $|x|\leq \epsilon$ if $|c-c_0|<\delta$\\
$O(1)$&$x=O(1)$ as $c \to c_0$, if $\exists \epsilon>0, \delta>0$ s.t., $|x|\leq \epsilon$ if $|c-c_0|<\delta$\\
$g(c)\sim h(c)$ as $c \to c_0$& $\lim_{c \to c_0} \frac{g(c)}{h(c)}=1$ or $g(c) =  h(c) (1+ o(1))$ as $c \to c_0$\\
$\{X_k\}$&Observation sequence\\
Stopping time $\tau$  on $\{X_k\}$& $\indic_{\{ \tau=n\}}=0 \text{ or } 1$ depends only on the values of $X_1, \ldots, X_n$\\
Change point $\Gamma, \gamma$ & Time index at which distribution of observations changes from $f_0$ to $f_1$\\
$\Prob_{n}$ ($\Expect_{n}$)& Probability measure (expectation) when the change occurs at time $n$\\
$\Prob_\infty$ ($\Expect_\infty$) & Probability measure (expectation) when the change does not occur\\
$\text{ess sup}\ X$ & Essential supremum of $X$, i.e., smallest $K$ such that $\Prob(X\leq K) =1$ \\
$D (f_1\| f_0)$& K-L Divergence between $f_1$ and $f_0$, defined as $\Expect_1\left(\log \frac{f_1 (X)}{f_0(X)} \right)$\\
$(x)^+$ & $\max\{x, 0\}$\\
$\ADD(\tau)$ & $\ADD(\tau)=\sum_{n=0}^\infty \Prob(\Gamma=n) \ \Expect_n \left[(\tau - \Gamma)^+\right]$ \\
$\PFA(\tau)$ & $\PFA(\tau)=\Prob (\tau<\Gamma) = \sum_{n=0}^\infty \ \Prob(\Gamma=n) \Prob_n (\tau<\Gamma)$\\
$\FAR(\tau)$ & $\FAR(\tau)=\frac{1}{\Expect_\infty[\tau]}$\\
$\WADD(\tau)$ & $\WADD(\tau)=\underset{n \geq 1}{\operatorname{\sup}}\ \text{ess sup} \ \Expect_n\left[(\tau-n)^+| X_1, \dots, X_{n-1}\right]$\\
$\CADD(\tau)$ & $\CADD(\tau)=\underset{n \geq 1}{\operatorname{\sup}}\ \Expect_n[\tau-n| \tau\geq n]$.\\
\hline
\end{tabular}}
 \end{table}

\section{Mathematical Preliminaries}

A typical  observation process will be denoted by sequence $\{X_n, n=1, 2, \ldots\}$. Before we describe the quickest change detection problem, we present some useful definitions and results that summarize the required mathematical background. For a detailed treatment of the topics discussed below we recommend \cite{will-book-probmart-1991}, \cite{chow-robb-sieg-book-1971}, \cite{sieg-seq-anal-book-1985} and \cite{wood-nonlin-ren-th-book-1982}.

\subsection{Martingales} \label{sec:martingales}
\begin{definition} The random sequence $\{X_n, n=1, 2, \ldots\}$ is called a {\em martingale} if $\Expect[X_n]$ is finite for all $n$,
and for any $k_1 < k_2 < \cdots < k_n < k_{n+1}$,
\begin{equation} \label{eq:martingale}
\Expect[X_{k_n+1} | X_{k_1} ,\ldots, X_{k_n}] = X_{k_n}
\end{equation}
If the ``$=$" in \eqref{eq:martingale} is replaced by ``$\leq$", then the sequence $\{X_n\}$ is called a {\em supermartingale},
and if the ``$=$" is replaced by ``$\geq$", the sequence is called a {\em submartingale}.
A martingale is both a supermartingale and a submartingale.
\end{definition}
Some important and useful results regarding martingales are as follows:
\begin{theorem}[\cite{chow-robb-sieg-book-1971}] (Kolmogorov's Inequality) Let $\{X_n, n=1, 2, \ldots\}$ be a {\em submartingale}. Then
\[
\Prob \left( \max_{1 \leq k \leq n} X_k ~\geq~ \gamma \right) \leq \frac{ \Expect [X_n^{+}]}{\gamma}, ~\forall ~\gamma > 0
\]
where $X_n^{+} = \max\{0, X_n\}$.
\end{theorem}
Kolmogorov's inequality can be considered to be a generalization of Markov's inequality, which is given by
\begin{equation} \label{eq:Markov_Ineq}
\Prob \left( X \geq \gamma\right) \leq \frac{ \Expect [X^{+}]}{\gamma}, ~\forall~ \gamma > 0
\end{equation}

As we will see in the following sections, quickest change detection procedures often involve comparing a stochastic sequence to a threshold to make decisions. Martingale inequalities often play a crucial role in the design of the threshold so that the procedure meets a false alarm constraint.
We now state one of the most useful results regarding martingales.

\begin{theorem}[[\cite{will-book-probmart-1991}] (Martingale Convergence Theorem) Let $\{X_n, n=1, 2, \ldots\}$  be a {\em martingale}
(or {\em submartingale} or {\em supermartingale}), such that $\sup_n \Expect[|X_n|] < \infty$.
Then, with probability one, the limit $X_\infty = \lim_{k \to \infty} X_n$ exists and is finite.
\end{theorem}

\subsection{Stopping Times} \label{sec:stopping_times}
\begin{definition} A {\em stopping time} with respect to the random sequence $\{X_n, n=1, 2, \ldots\}$ is a random
variable $\tau$ with the property that for each $n$, the event $\{ \tau=n\} \in \sigma (X_1, \ldots, X_n)$,
where $\sigma (X_1, \ldots, X_n)$ denotes the sigma-algebra generated by $(X_1, \ldots, X_n)$.
Equivalently, the random variable $\indic_{\{ \tau=n\}}$,
which is the indicator of the event $\{ \tau =n\}$,  is a function of only $X_1, \ldots, X_n$.
\end{definition}

Sometimes the definition of a stopping time $\tau$ also requires that $\tau$ be finite almost surely, i.e., that $\Prob (\tau < \infty) = 1$.

Stopping times are essential to sequential decision making procedures such as quickest change detection procedures, since the times at which decisions are made are stopping times with respect to the observation sequence. There are two main results concerning stopping times that are of interest.
\begin{theorem}[\cite{chow-robb-sieg-book-1971}]\label{thm:DoobOptSamp} (Doob's Optional Stopping Theorem) Let $\{X_n, n=1, 2, \ldots\}$  be a
{\em martingale}, and let $\tau$ be a stopping time with respect to $\{X_n, n=1, 2, \ldots\}$. If the following conditions hold:
\begin{enumerate}
\item $\Prob( \tau < \infty) = 1$.
\item $\Expect[|X_\tau|] < \infty$.
\item $\Expect [ X_n \indic_{\{\tau > n\}} ] \to 0$ as $n \to \infty$.
\end{enumerate}
then
\[
\Expect[X_\tau] = \Expect[X_1].
\]
Similarly, if the above conditions hold, and if $\{X_n, n=1, 2, \ldots\}$  is a {\em submartingale}, then
\[
\Expect[X_\tau] \geq \Expect[X_1],
\]
and if $\{X_n, n=1, 2, \ldots\}$  is a {\em supermartingale}, then
\[
\Expect[X_\tau] \leq \Expect[X_1].
\]
\end{theorem}
\begin{theorem}[\cite{sieg-seq-anal-book-1985}] \label{th:Wald_Identity} (Wald's Identity) Let $\{X_n, n=1, 2, \ldots\}$ be a sequence of independent and
identically distributed (i.i.d.) random variables, and let $\tau$ be a stopping time with respect to
$\{X_n, n=1, 2, \ldots\}$. Furthermore, define the sum at time $n$ as
\[
S_n = \sum_{k=1}^n X_k
\]
Then, if $\Expect[|X_1|] < \infty$ and $\Expect [\tau] < \infty$,
\[
\Expect[ S_\tau] = \Expect[X_1] \, \Expect[\tau]
\]
\end{theorem}
Like martingale inequalities, the optional stopping theorem is useful in the false alarm analysis of quickest change detection procedures. Both the optional stopping theorem and Wald's identity also play a key role in the delay analysis of quickest change detection procedures.
%
%
\subsection{Renewal and Nonlinear Renewal Theory}
\label{sec:renewalTheory}
As we will see in subsequent sections, quickest change detection procedures often involve comparing a stochastic sequence to a threshold to make decisions. Often the stochastic sequence used in decision-making can be expressed as a sum of a random walk and possibly a \textit{slowly changing} perturbation. To obtain accurate estimates of the performance of the detection procedure, one needs to obtain an accurate estimate of the distribution of the overshoot of the stochastic sequence when it crosses the decision threshold.
Under suitable assumptions, and when the decision threshold is large enough, the overshoot distribution of the stochastic sequence can be approximated by the overshoot distribution of the random walk. It is then of interest to have asymptotic estimates of the overshoot distribution, when a random walk crosses a large boundary.

Consider a sequence of i.i.d. random variables $\{Y_n\}$ (with $Y$ denoting a generic random variable in the sequence) and let
\[S_n = \sum_{k=1}^n Y_k, \]
and
\[\tau = \inf\{n\geq 1: S_n > b\}.\]
The quantity of interest is the distribution of the overshoot $S_\tau -b$.
If $\{Y_n\}$ are i.i.d. \textit{positive} random variables with cumulative distribution function (c.d.f.) $F(y)$,
then  $\{Y_n\}$ can be viewed as inter-arrival times of buses at a stop. The overshoot is then the time to next bus when an observer is waiting for a bus at time $b$.
The distribution of the overshoot, and hence also of the time to next bus, as $b\to \infty$ is a well known result
in renewal theory.
\begin{theorem}[\cite{sieg-seq-anal-book-1985}]  \label{thm:ren_th}
If $\{Y_n\}$ are nonarithmetic\footnote{A random variable is arithmetic if all of it probability mass is on a lattice. Otherwise it is said to non-arithmetic.} random variables, and $\Prob(Y>0)=1$, then
\[\lim_{b\to \infty} \Prob(S_\tau-b > y) = (\Expect[Y])^{-1} \int_y^\infty \Prob \{ Y > x\} dx.\]
Further, if $\Expect[Y^2]<\infty$, then
\[\lim_{b\to \infty} \Expect(S_\tau-b) = \frac{\Expect[Y^2]}{2\Expect[Y]}.\]
\end{theorem}
\vspace{0.5cm}

When the $\{Y_n\}$ are i.i.d. but not necessarily non-negative, and $\Expect[Y]>0$, then the following
concept of {\em ladder variables} can be used. Let
\[\tau_{+} = \inf\{n\geq 1: S_n > 0\}.\]
Note that if $\tau_{+}< \infty$, then $S_{\tau_{+}}$ is a positive random variable. Also, if
\[
\tau = \inf\{n\geq 1: S_n > b\} < \infty
\]
then the distribution of $S_\tau-b$ is the same as the overshoot distribution for
the sum of a sequence of i.i.d. positive random variables (each with distribution equal to that of $S_{\tau_{+}}$) crossing the boundary $b$. Therefore, by applying  Theorem~\ref{thm:ren_th}, we have the following result.
\begin{theorem}[\cite{sieg-seq-anal-book-1985}] \label{thm:ren_th2}
If $\{Y_n\}$ are nonarithmetic, then
\[\lim_{b\to \infty} \Prob(S_\tau-b > y) = (\Expect[S_{\tau_{+}}])^{-1} \int_y^\infty \Prob(S_{\tau_{+}}>x)\, dx.\]
Further, if $\Expect[Y^2]<\infty$, then
\[\lim_{b\to \infty} \Expect(S_\tau-b) = \frac{\Expect[S_{\tau_{+}}^2]}{2\Expect[S_{\tau_{+}}]}.\]
\end{theorem}
%
Techniques for computing the required quantities involving the distribution of the ladder height $S_{\tau_{+}}$ in Theorem~\ref{thm:ren_th2} can be found in \cite{sieg-seq-anal-book-1985}.

As mentioned earlier, often the stochastic sequence considered in quickest change detection problem can be
written as a sum of a random walk and a sequence of small perturbations.
Let
\[Z_n = \sum_{k=1}^{n} Y_k + \eta_n,\]
and
\[\tau = \inf\{n\geq 1: Z_n > b\}.\]
Then,
\[Z_\tau = \sum_{k=1}^{\tau} Y_k + \eta_\tau.\]
Therefore, assuming that $\Expect[\tau] < \infty$,  Wald's Identity (see Theorem~\ref{th:Wald_Identity}) implies that
\begin{eqnarray}
\Expect[Z_\tau] &=& \Expect\left[\sum_{k=1}^{\tau} Y_k \right] + \Expect[\eta_\tau].\\
                &=& \Expect[\tau] \Expect[Y] + \Expect[\eta_\tau].
\end{eqnarray}
Thus,
\begin{eqnarray*}
\Expect[\tau]  &=& \frac{\Expect[Z_\tau] -\Expect[\eta_\tau]}{\Expect[Y]} \\
               &=& \frac{b + \Expect[Z_\tau-b] - \Expect[\eta_\tau]}{\Expect[Y]}.
\end{eqnarray*}
If $\Expect[\eta_\tau]$ and $\Expect[Z_\tau-b]$ are finite then it is easy to see that
\[
\Expect[\tau] \sim \frac{b}{\Expect[Y]}  \text{ as } b\to \infty
\]
where $\sim$ is as defined in Table~\ref{Tab:Glossary}.

But if we can characterize the overshoot distribution of $\{Z_n\}$ when it crosses a large threshold then we can obtain
better approximations for $\Expect[\tau]$.
Nonlinear renewal theory allows us to obtain distribution of the overshoot when $\{\eta_n\}$ satisfies some properties.
\begin{definition}
$\{\eta_n\}$ is  a {\em slowly changing} sequence if
\begin{equation}
\label{eq:SlowlyCngCond1}
n^{-1} \max\{ |\eta_1|, \ldots, |\eta_n|\} \xrightarrow[i.p.]{n\to\infty} 0,
\end{equation}
and for every $\epsilon > 0$, there exists $n^*$ and $\delta>0$ such that for all $n\geq n^*$
\begin{equation}
\label{eq:SlowlyCngCond2}
\Prob \left(\max_{1\leq k\leq n\delta} |\eta_{n+k} - \eta_{n}| > \epsilon\right) < \epsilon.
\end{equation}
\end{definition}
If indeed $\{\eta_n\}$ is a slowly changing sequence, then the distribution of $Z_{\tau}- b$, as $b\to \infty$,
is equal to the asymptotic distribution of the overshoot
when the random walk $S_n = \sum_{k=1}^{n} Y_k$ crosses a large positive boundary, as stated in the following result.

\begin{theorem}[\cite{sieg-seq-anal-book-1985}]
\label{thm:NonlinearRenewalTheort}
If $\{Y_n\}$ are nonarithmetic and $\{\eta_n\}$ is a slowly changing sequence then
\[\lim_{b\to \infty} \Prob(Z_\tau-b \leq x) = \lim_{b\to \infty} \Prob(S_\tau-b \leq x).\]
Further, if $\mathrm{Var}(Y)<\infty$ and certain additional conditions ((9.22)-(9.27) in \cite{sieg-seq-anal-book-1985})
are satisfied, then
\begin{eqnarray*}
\Expect[\tau]  = \frac{b + \zeta - \Expect[\eta]}{\Expect[Y]} + o(1)~\text{as $b \to \infty$}
\end{eqnarray*}
where $\zeta=\frac{\Expect[S_{\tau_{+}}^2]}{2\Expect[S_{\tau_{+}}]}$, and $\eta$ is the limit of $\{\eta_n\}$ in
distribution.
\end{theorem}

\section{Bayesian quickest change detection}
\label{sec:BayesCent}
As mentioned earlier we will primarily focus on the case where the observation process $\{X_n\}$
is a discrete time stochastic process, with $X_n$ taking real values,
 whose distribution changes at some unknown change point.
In the Bayesian setting it is assumed that the change point
is a random variable $\Gamma$ taking values on the non-negative integers,
with $\pi_n = \Prob\{\Gamma =n\}$.
Let $\Prob_{n}$ (correspondingly $\Expect_{n}$) be the probability measure (correspondingly expectation) when the change
occurs at time $\tau=n$. Then,
$\Prob_\infty$ and $\Expect_\infty$ stand for the probability measure and expectation
when $\tau=\infty$, i.e., the change does not occur.
At each time step a decision is made based on all the information available as to whether to stop
and declare a change or to continue taking observations.
Thus the time at which the change is declared is a stopping time $\tau$ on
the sequence $\{X_n\}$ (see Section~\ref{sec:stopping_times}).
Define the average detection delay (ADD) and the probability of false alarm (PFA),  as
\begin{eqnarray}
\ADD(\tau) &=& \Expect \left[(\tau - \Gamma)^+\right] = \sum_{n=0}^\infty \pi_n \Expect_n   \left[(\tau - \Gamma)^+\right] \label{eq:ADD_def}\\
\PFA(\tau) &=& \Prob (\tau<\Gamma) = \sum_{n=0}^\infty \pi_n \Prob_n (\tau<\Gamma)
\end{eqnarray}
Then, the Bayesian
quickest change detection problem is to minimize $\ADD$  subject to a constraint on $\PFA$.  Define the class of stopping times that satisfy a constraint $\alpha$ on $\PFA$:
\begin{equation} \label{eq:Cclassdef}
\Cc_\alpha=\{\tau: \PFA(\tau) \leq \alpha\}.
\end{equation}
Then the Bayesian quickest change detection problem as formulated by Shiryaev is as follows.
%
\begin{equation}
\label{eq:BayesConstProb}
\text{\emph{Shiryaev's Problem:} For a given $\alpha$, find a stopping time $\tau \in \Cc_\alpha$  to minimize $\ADD (\tau)$.}
\end{equation}
Under an i.i.d. model for the observations, and a geometric model for the change point $\Gamma$, \eqref{eq:BayesConstProb} can be solved exactly by relating it to a stochastic control problem  \cite{shir-siamtpa-1963, shir-opt-stop-book-1978}.
We discuss this  i.i.d. model  in detail in Section~\ref{sec:IIDBayesianSetting}. When the model is not i.i.d., it is difficult to find algorithms that are exactly optimal. However, asymptotically optimal solutions, as $\alpha \to 0$,
are available in a very general non-i.i.d. setting \cite{tart-veer-siamtpa-2005}, as discussed in Section~\ref{sec:NONIIDBayes}.

\subsection{The Bayesian I.I.D. Setting}
\label{sec:IIDBayesianSetting}

Here it is assumed that conditioned on the change point $\Gamma$, the random variables $\{X_n\}$ are i.i.d.\ with probability density
function (p.d.f.) $f_0$ before the change point, and i.i.d.\ with p.d.f. $f_1$ after the change point.
The change point $\Gamma$ is modeled as {\em geometric} with parameter $\rho$, i.e., for $0 < \rho < 1$
\begin{equation} \label{eq:geom_prior}
\pi_n  = \Prob \{ \Gamma= n \} =  \rho (1-\rho)^{n-1} \; \indic_{\{n\geq 1\}}, \quad \pi_0 =0
\end{equation}
where $\indic$ is the indicator function. The goal is to choose a stopping time $\tau$ on the
observation sequence $\{X_n\}$ to solve \eqref{eq:BayesConstProb}.

A solution to  \eqref{eq:BayesConstProb} is provided in Theorem \ref{thm:ShiryaevOpt} below.
Let $X_1^n = (X_1, \ldots, X_n)$ denote the observations up to time $n$. Also let
\begin{equation} \label{eq:pk}
p_n = \Prob ( \Gamma \leq n \; | \; X_1^n )
\end{equation}
be the \textit{a posteriori} probability at time $n$ that the change has taken place given the observation up to time $n$.
Using Bayes' rule, $p_n$ can be shown to satisfy the recursion
\begin{equation} \label{eq:pkrec}
p_{n+1} = \Phi(X_{n+1}, p_n)
\end{equation}
where
\begin{equation} \label{eq:recursionTake}
\Phi(X_{n+1},p_n)  =  \frac{\tilde{p}_n   L(X_{n+1})}{ \tilde{p}_n  L(X_{n+1}) + (1-\tilde{p}_n )}
\end{equation}
$\tilde{p}_n = p_n + (1-p_n) \rho$,
$L(X_{n+1}) = f_1(X_{n+1})/f_0(X_{n+1})$ is the likelihood ratio,
and $p_0=0$.

\begin{definition} \label{def:K-L-Divergence} (Kullback-Leibler (K-L) Divergence). The K-L divergence between two p.d.f.'s $f_1$ and $f_0$ is defined as
\[
D (f_1\| f_0) = \int  f_1 (x) \;   \log \frac{f_1 (x)}{f_0(x)} \; dx.
\]
\end{definition}

Note that $D (f_1\| f_0) \geq 0$ with equality iff $f_1 = f_0$ almost surely. We will assume that
\[
0 < D(f_1\| f_0) < \infty.
\]
\begin{theorem}[\cite{shir-siamtpa-1963,shir-opt-stop-book-1978}]\label{thm:ShiryaevOpt}
The optimal solution to Bayesian optimization problem of \eqref{eq:BayesConstProb} is the \emph{Shiryaev algorithm/test}, which is described by the stopping time:
\begin{equation}
\label{eq:OptimalAlgo}
\taus = \inf\left\{ n\geq 1: p_n \geq A_{\alpha} \right\}
\end{equation}
if  $A_\alpha \in (0,1)$ can be chosen such that
\begin{equation} \label{eq:equality_constraint}
\PFA (\taus)  = \alpha.
\end{equation}
\end{theorem}
\begin{proof}
Towards solving \eqref{eq:BayesConstProb}, we consider a Lagrangian relaxation of this problem that can be solved  using dynamic programming.
\begin{equation}
\label{eq:Lagrangian}
J^* = \min_{\tau} \Big(\Expect\left[(\tau - \Gamma)^+\right] + \lambda_f \; \Prob(\tau<\Gamma)\Big)
\end{equation}
where $\lambda_f$ is the Lagrange multiplier, $\lambda_f\geq 0$.
It is shown in \cite{shir-siamtpa-1963, shir-opt-stop-book-1978} that under the assumption \eqref{eq:equality_constraint}, there exists a $\lambda_f$ such that the solution to \eqref{eq:Lagrangian}
is also the solution to \eqref{eq:BayesConstProb}.

Let $\Theta_n$ denote the state of the system at time $n$. After the stopping time $\tau$ it
is assumed that the system enters a terminal state ${\cal T}$ and stays there.
For $n < \tau$, we have  $\Theta_n =0$ for $n < \Gamma$, and $\Theta_n = 1$ otherwise. Then we can write
\[
\ADD(\tau) = \Expect\left[ \sum_{n=0}^{\tau-1} \indic_{\{\Theta_n= 1\}} \right] \ \ \  \ \ \text{  and  } \ \ \ \ \ \PFA(\tau) = \Expect [ \indic_{\{\Theta_\tau = 0\}}].
\]
Furthermore, let $D_n$ denote the stopping decision variable at time $n$, i.e., $D_n = 0$ if $k < \tau$ and $D_n = 1$ otherwise.
Then the optimization problem in \eqref{eq:Lagrangian} can be written as a minimization of an additive cost over time:
\[
J^* = \min_{\tau}  \Expect\left[\sum_{n=0}^{\tau}  g_n (\Theta_n, D_n)\right]
\]
with
\[
g_n (\theta, d) = \indic_{\{\theta \neq {\cal T}\} } \; \left[  \indic_{\{\theta=1\}} \indic_{\{d=0\}} + \lambda_f \; \indic_{\{\theta=0\}} \indic_{\{d=1\}} \right].
\]
Using standard arguments \cite{bert-dyn-prog-book-1995} it can be seen that this optimization problem
can be solved using infinite horizon dynamic programming with sufficient statistic (belief state) given by:
\[
\Prob( \Theta_n = 1 \; | \; X_1^n) = \Prob ( \Gamma \leq n \; | \; X_1^n) = p_n
\]
which is the {\em a posteriori probability} of \eqref{eq:pk}.

The optimal policy for the problem given in \eqref{eq:Lagrangian} can be obtained from the solution to the Bellman equation:
\begin{equation} \label{eq:Bellman}
J(p_n) = \min_{d_n}  \lambda_f \; (1-p_n) \indic_{\{d_n=1\}} + \indic_{\{d_n=0\}} \left[ p_n + A_J (p_n)\right]
\end{equation}
where
\[
A_J (p_n) = \Expect[J (\Phi(X_{n+1}, p_n))].
\]
It can be shown by using an induction argument that both $J$ and $A_J$ are non-negative
concave functions on the interval $[0,1]$,  and that $J(1) = A_J (1) =0$.
Then, it is easy to show that the optimal solution for the problem
in \eqref{eq:Lagrangian} has the following structure:
\[
\taus = \inf\left\{ k\geq 1: p_n \geq A \right\}.
\]
\end{proof}
See Fig. \ref{fig:BellmanCurvesForShiryaev} for a plot of $\lambda_f (1-p_n)$ and $p_n + A_J(p_n)$ as a function of  $p_n$. Fig. \ref{fig:Shiryaev_ADD_PFA_Evolution} shows a typical evolution of the optimal Shiryaev algorithm.
\begin{figure}[!tb]
  \centering
  \subfloat[A plot of the cost curves for $\lambda_f=10$,  $\rho=0.01$, $f_0 \sim {\cal N}(0,1)$, $f_1 \sim {\cal N}(0.75,1)$]
  {\label{fig:BellmanCurvesForShiryaev}\includegraphics[width=0.45\textwidth]{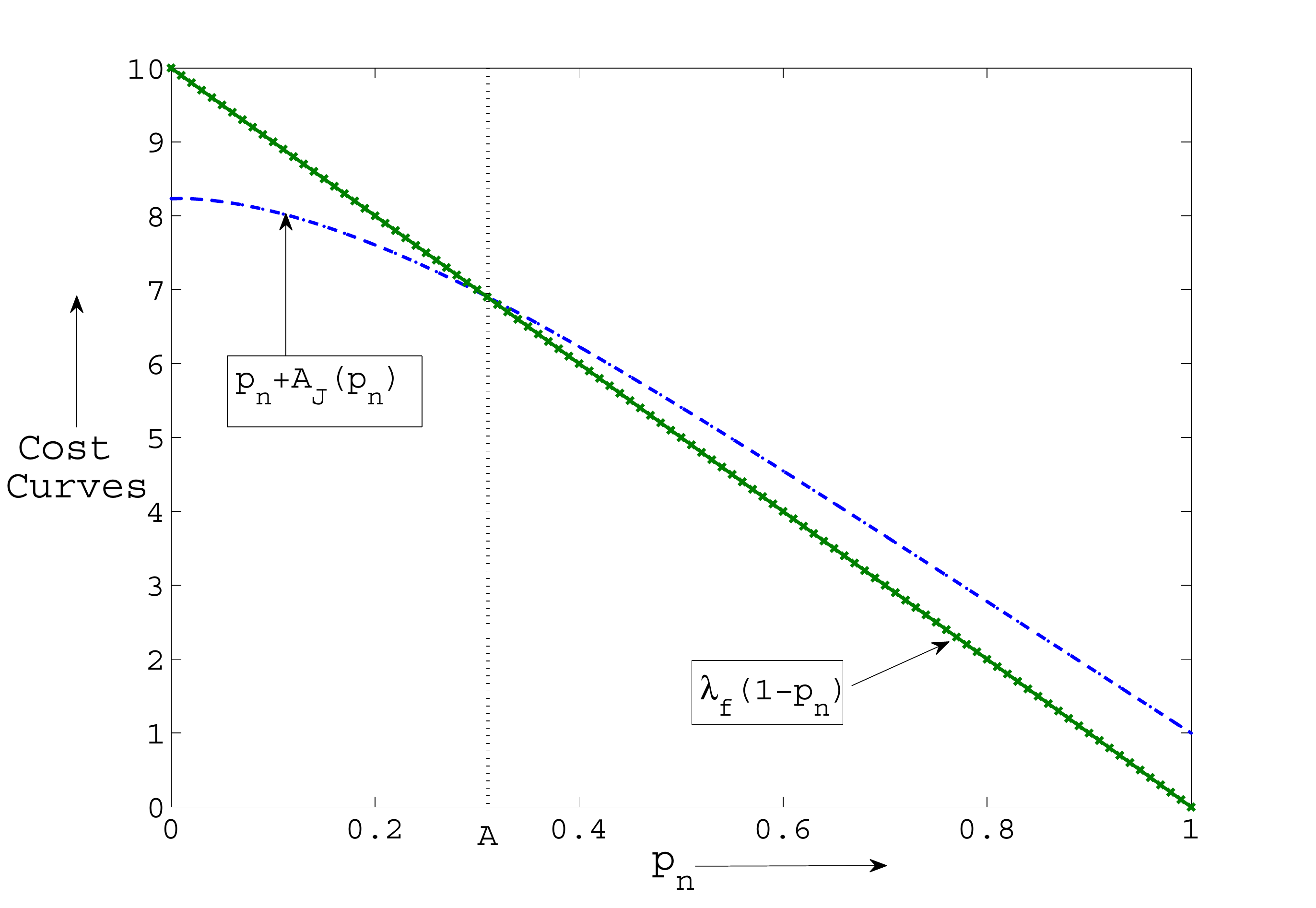}} \hspace{0.5cm}
  \subfloat[Typical Evolution of the Shiryaev algorithm. Threshold $A=0.99$ and change point $\Gamma=100$.]{\label{fig:Shiryaev_ADD_PFA_Evolution}\includegraphics[width=0.45\textwidth]{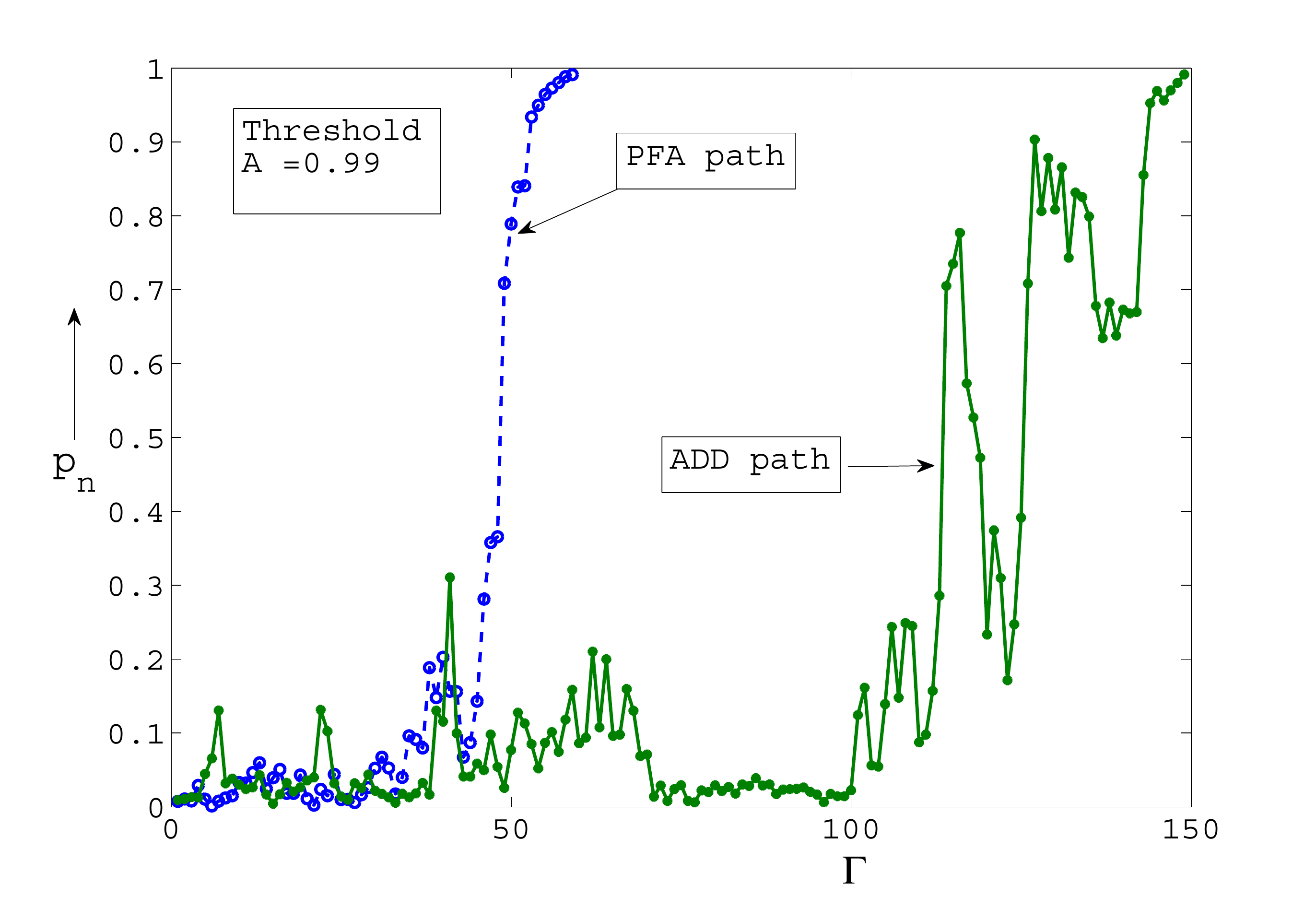}}
    \caption{Cost curves and typical evolution of the Shiryaev algorithm.\label{fig:CostCurvesAndEvolutionforShiryaev2}
}
 \end{figure}

We now discuss some alternative descriptions of the Shiryaev algorithm. Let
\[
\Lambda_n  = \frac{p_n}{(1-p_n)}
\]
and
\[
R_{n,\rho}  = \frac{p_n}{(1-p_n)\rho}.
\]
We note that $\Lambda_n$ is the likelihood ratio of the hypotheses `$H_1: \Gamma \leq n$' and `$H_0: \Gamma > n$' averaged over the change point:
\begin{eqnarray}
\Lambda_n & = & \frac{p_n}{(1-p_n)} \nonumber\\ & = & \frac{\Prob(\Gamma\leq n | X_1^n)}{\Prob(\Gamma> n | X_1^n)} \nonumber\\
&=& \frac{\sum_{k=1}^n (1-\rho)^{k-1} \rho \prod_{i=1}^{k-1} f_0(X_i) \prod_{i=k}^{n} f_1(X_i)}{(1-\rho)^n \prod_{i=1}^n f_0(X_i)}\nonumber\\
&=& \frac{1}{(1-\rho)^n} \sum_{k=1}^n (1-\rho)^{k-1} \rho \prod_{i=k}^{n} L(X_i) .
\label{eq:ShiryaevLambdanExpanded}
\end{eqnarray}
where as defined before $L(X_i) = \frac{f_1(X_i)}{f_0(X_i)}$.
Also, $R_{n,\rho}$ is a scaled version of $\Lambda_n$:
\begin{equation}\label{eq:ShiryaevRnrhoExpanded}
R_{n,\rho} = \frac{1}{(1-\rho)^n} \sum_{k=1}^n (1-\rho)^{k-1} \prod_{i=k}^{n} L(X_i) .
\end{equation}
Like $p_n$, $R_{n,\rho}$ can also be computed using a recursion:
\[R_{n+1,\rho} = \frac{1+R_{n,\rho}}{1-\rho} L (X_{n+1}), \quad R_{0,\rho} = 0.\]
It is easy to see that $\Lambda_n$ and $R_{n,\rho}$ have one-to-one mappings with the Shiryaev statistic $p_n$.
\begin{algorithm}[Shiryaev Algorithm] \label{algo:Shiryaev} The following three stopping times are equivalent and define the Shiryaev stopping time.
 \begin{equation}\label{eq:Shiryaevalgo}
 \taus = \inf\{ n \geq 1: p_n \geq A\}
 \end{equation}
\begin{equation}\label{eq:ShiryaevLambdan}
\taus = \inf\{ n \geq 1: \Lambda_n \geq a\}
\end{equation}
\begin{equation}\label{eq:ShiryaevRnrho}
\taus = \inf\{ n \geq 1: R_{n,\rho} \geq \frac{a}{\rho}\}
\end{equation}
with $a=\frac{A}{1-A}$
\end{algorithm}
We will later see that defining the Shiryaev algorithm using the statistic $\Lambda_n$
\eqref{eq:ShiryaevLambdanExpanded} will be useful in Section~\ref{sec:NONIIDBayes},
where we discuss the Bayesian quickest change detection problem in a  non-i.i.d. setting.
Also, defining the Shiryaev algorithm using the statistic $R_{n,\rho}$
\eqref{eq:ShiryaevRnrhoExpanded} will be useful in
Section~\ref{sec:MinimaxCent} where we discuss quickest change detection in a minimax setting.

\subsection{General Asymptotic Bayesian Theory}
\label{sec:NONIIDBayes}
As mentioned earlier, when the observations are not i.i.d.\ conditioned on the change point $\Gamma$, then finding an exact solution to the problem \eqref{eq:BayesConstProb} is difficult. Fortunately, a Bayesian {\em asymptotic} theory can be developed for quite general pre- and post- change distributions \cite{tart-veer-siamtpa-2005}. In this section we discuss the results from \cite{tart-veer-siamtpa-2005} and provide a glimpse of the proofs.

We first state the observation model studied in \cite{tart-veer-siamtpa-2005}. When the process evolves in the pre-change regime, the conditional density of $X_n$ given $X_1^{n-1}$ is $f_{0,n}(X_n|X_1^{n-1})$. After the change happens, the conditional density of $X_n$ given $X_1^{n-1}$  is given by $f_{1,n}(X_n|X_1^{n-1})$.

As in the i.i.d.\ case, we can define the {\em a posteriori} probability of change having taken place before time $n$, given the observation up to time $n$, i.e.,
\begin{equation} \label{eq:pk2}
p_n = \Prob ( \Gamma \leq n \; | \; X_1^n )
\end{equation}
with the understanding that the recursion \eqref{eq:pkrec} is no longer valid, except for the i.i.d. model.


We note that in the non-i.i.d. case also $\Lambda_n=\frac{p_n}{1-p_n}$ is the
likelihood ratio of the hypotheses ``$H_1: \Gamma \leq n$'' and ``$H_0: \Gamma > n$''.
If $\pi_n = \Prob \{\Gamma= n\}$, then following \eqref{eq:ShiryaevLambdanExpanded}, $\Lambda_n$ can be written for a general
change point distribution $\{\pi_n\}$ as
\[
\begin{split}
\Lambda_n  & = \frac{p_n}{(1-p_n)} \\
& = \frac{\Prob(\Gamma\leq n | X_1^n)}{\Prob(\Gamma> n | X_1^n)}\\
&= \frac{\sum_{k=1}^n \pi_n \prod_{i=1}^{k-1} f_{0,i}(X_i|X_1^{i-1}) \prod_{i=k}^{n} f_{1,i}(X_i|X_1^{i-1})}{\Prob(\Gamma>n) \prod_{i=1}^n f_{0,i}(X_i|X_1^{i-1})}\\
&= \frac{1}{\Prob(\Gamma>n)} \sum_{k=1}^n \pi_n \prod_{i=k}^{n} Y_i
\end{split}
\]
where
\[
Y_i = \log \frac{f_{1,i}(X_i|X_1^{i-1})}{f_{0,i}(X_i|X_1^{i-1})}.
\]
If $\Gamma$ is geometrically distributed with parameter $\rho$, the above expression reduces to
\begin{eqnarray*}
\Lambda_n = \frac{1}{(1-\rho)^n} \sum_{k=1}^n (1-\rho)^{k-1} \rho \prod_{i=k}^{n} Y_i.
\end{eqnarray*}
In fact, $\Lambda_n$ can even be computed recursively in this case:
\begin{eqnarray}
\label{eq:Lambda_n}
\Lambda_{n+1} = \frac{1 + \Lambda_n}{1-\rho} Y_{n+1}
\end{eqnarray}
with $\Lambda_0=0$.

In \cite{tart-veer-siamtpa-2005}, it is shown that if there exists $q$ such that
\begin{equation}
\label{eq:noniidCond_exist_q}
\frac{1}{t}\sum_{i=n}^{n+t} Y_i \to q \ \ \ \text{ a.s. }\Prob_n \ \ \text{ when } t \to \infty  \ \ \forall n
\end{equation}
($q=D(f_1 || f_0)$ for the i.i.d. model),
 and some additional conditions on the rates of convergence are satisfied, then the
 Shiryaev algorithm \eqref{eq:Shiryaevalgo} is asymptotically optimal for Bayesian optimization problem of \eqref{eq:BayesConstProb} as $\alpha \to 0$.
In fact, $\taus$ minimizes all moments of the detection delay as well as the moments of the delay, conditioned on the change point.
The asymptotic optimality proof is based on first finding a lower bound on the asymptotic moment of the delay of all
the detection procedures in the class $\Cc_\alpha$, as $\alpha \to 0$, and then showing that the
Shiryaev stopping time \eqref{eq:Shiryaevalgo} achieves that lower bound asymptotically.

To state the theorem, we need the following definitions. Let $q$ be the limit as specified in
\eqref{eq:noniidCond_exist_q}, and let $0< \epsilon < 1$. Then define
\[{T_\epsilon}^{(n)} = \sup \left\{ t\geq 1: \Big| \frac{1}{t}\sum_{i=n}^{n+t} Y_i - q \Big| > \epsilon \right\}. \]
Thus, ${T_\epsilon}^{(n)}$ is the last time that the log likelihood sum $\sum_{i=n}^{n+t} Y_i$
falls outside an interval of length $\epsilon$ around $q$.
In general, existence of the limit $q$ in \eqref{eq:noniidCond_exist_q}
only guarantees $\Prob_n({T_\epsilon}^{(n)}< \infty)=1$, and not the finiteness of the
moments of ${T_\epsilon}^{(n)}$.
Such conditions are needed for existence of moments of detection delay of $\taus$. In particular, for some $r \geq 1$, we need:
\begin{equation}
\label{eq:noniidCond_nmoment_exists}
\Expect_n[{T_\epsilon}^{(n)}]^r < \infty \text{ for all } \ \epsilon > 0 \text{ and } \ n \geq 1,
\end{equation}
and
\begin{equation}
\label{eq:noniidCond_Avg_moment_exists}
\sum_{n=1}^{\infty} \pi_n \Expect_n[{T_\epsilon}^{(n)}]^r < \infty \text{ for all } \ \epsilon > 0.
\end{equation}
Now, define
\[
d = - \lim_{n\to \infty} \frac{\log \Prob(\Gamma > n)}{n}.
\]
The parameter $d$ captures the tail parameter of the distribution of $\Gamma$. If $\Gamma$ is `heavy tailed'
then $d=0$, and if $\Gamma$ has an `exponential tail' then $d>0$. For example, for the geometric prior with parameter $\rho$, $d=|\log(1-\rho)|$.

\begin{theorem}[\cite{tart-veer-siamtpa-2005}]
\label{thm:BayesianAsympVVV}
 If the likelihood ratios are such that \eqref{eq:noniidCond_exist_q} is satisfied then
\begin{enumerate}
\item If $a= a_\alpha = \frac{1-\alpha}{\alpha}$, then $\taus$ as defined in \eqref{eq:Shiryaevalgo} belongs to the set $\Cc_\alpha$.
\item For all $n \geq 1$, the $m$-th moment of the conditional delay, conditioned on $\Gamma=n$, satisfies:
\begin{equation}
\label{eq:noniid_LB1}
\underset{\tau\in \Cc_\alpha}{\operatorname{\inf}}\ \ \Expect_n[(\tau-n)^{+}]^{m}
\geq \left(\frac{|\log\alpha |}{q + d}\right)^m \Big( 1 + o(1) \Big) \text{ as } \alpha \to 0.
\end{equation}
\item For all $n \geq 1$, if \eqref{eq:noniidCond_nmoment_exists} is satisfied then for all $m\leq r$,
\begin{equation}
\label{eq:noniid_UB1}
\begin{split}
\Expect_n[(\taus-n)^{+}]^{m}
& \leq \left(\frac{\log a }{q + d}\right)^m \Big( 1 + o(1) \Big) \text{ as } a \to \infty\\
& = \left(\frac{|\log\alpha |}{q + d}\right)^m \Big( 1 + o(1) \Big) \text{ as } \alpha \to 0 \text{~if $a = a_\alpha = \frac{1-\alpha}{\alpha}$.}
\end{split}
\end{equation}

 \item The $m$-th (unconditional) moment of the delay satisfies
\begin{equation}
\label{eq:noniid_LB2}
\underset{\tau\in \Cc_\alpha}{\operatorname{\inf}}\ \ \Expect[(\tau-\Gamma)^{+}]^{m}
\geq \left(\frac{|\log\alpha|}{q + d}\right)^m \Big( 1 + o(1) \Big) \text{ as } \alpha \to 0.
\end{equation}
\item If \eqref{eq:noniidCond_nmoment_exists} and \eqref{eq:noniidCond_Avg_moment_exists} are satisfied,
then for all $m\leq r$
\begin{equation}
\label{eq:noniid_UB2}
\begin{split}
\Expect[(\taus-\Gamma)^{+}]^{m}
& \leq \left(\frac{\log a}{q + d}\right)^m \Big( 1 + o(1) \Big) \text{ as } a \to \infty\\
& = \left(\frac{|\log \alpha|}{q + d}\right)^m \Big( 1 + o(1) \Big) \text{ as } \alpha \to 0 \text{~if $a=a_\alpha = \frac{1-\alpha}{\alpha}$}.
\end{split}
\end{equation}
\end{enumerate}
\end{theorem}
\smallskip

\begin{proof} We provide  sketches of the proofs for  part 1), 2) and 3). The proofs of 4) and 5) follow by averaging the results in 2) and 3) over the prior on the change point.
\begin{enumerate}
\item Note that
\[\PFA(\taus) = \Expect[1-p_{\taus}] \leq 1-A_\alpha.\]
Thus, $A_\alpha=1-\alpha$ would ensure $\PFA(\taus)\leq \alpha$. Since,
$a_\alpha = \frac{A_\alpha}{1-A_\alpha}$, we have the result.
\item 
Let $L_\alpha$ be a positive number. By Chebyshev inequality,
\[\Prob_n((\tau-n)^m > L_\alpha^m) \leq \frac{\Expect_n[(\tau-n)^{+}]^{m}}{L_\alpha^m}.\]
This gives a lower bound on the detection delay
\begin{eqnarray*}
\Expect_n[(\tau-n)^{+}]^{m} &\geq& L_\alpha^m \ \Prob_n((\tau-n)^m > L_\alpha^m)\\
&=& L_\alpha^m \ \Prob_n(\tau-n > L_\alpha).
\end{eqnarray*}
Minimizing over the family $C_\alpha$, we get
\begin{eqnarray*}
\inf_{\tau\in \Cc_\alpha} \; \Expect_n[(\tau-n)^{+}]^{m}
\geq L_\alpha^m \ \left[\underset{\tau\in \Cc_\alpha}{\operatorname{\inf}}\Prob_n(\tau-n > L_\alpha)\right].
\end{eqnarray*}
Thus, if
\begin{equation}\label{eq:noniddCond_supProb}
\inf_{\tau\in \Cc_\alpha} \; \Prob_n(\tau-n > L_\alpha) \to 1 \text{ as } \alpha \to 0
\end{equation}
then $L_\alpha^m$ is a lower found for the detection delay of the family $\Cc_\alpha$.
It is shown in \cite{tart-veer-siamtpa-2005}
that if  \eqref{eq:noniidCond_exist_q} is satisfied then \eqref{eq:noniddCond_supProb} is true for
$ L_\alpha = (1-\epsilon) \frac{|\log\alpha|}{q + d}$ for all $\epsilon > 0$.
\item We only summarize the technique used to obtain the upper bound.
Let $\{S_n\}$ be any stochastic process such that
\[
\frac{S_n}{n} \to q \text{ as } n \to \infty.
\]
Let
\[
\eta = \inf\{n\geq 1: S_n > b\}
\]
and for $\epp > 0$
\[
T_\epp = \sup\{n\geq 1: \Big| \frac{S_n}{n}- b \Big| > \epp\}.
\]
First note that $S_{\eta-1} < b < S_\eta$. Also, on the set $\{k \geq T_\epp\}$,
$|S_n/n-q| < \epp$ for all $n\geq k$. The event $\{|S_n/n-q| < \epp\}$ implies $n \leq \frac{S_n}{q-\epp}$.
Using these observations we have
\begin{eqnarray*}
\eta-1 &=& (\eta-1) \indic_{\{\eta-1>T_\epp\}} + (\eta-1) \indic_{\{\eta-1 \leq T_\epp\}}\\
       &\leq& (\eta-1) \indic_{\{\eta-1>T_\epp\}} + T_\epp\\
       &\leq& \frac{b}{q-\epp}  + T_\epp.
\end{eqnarray*}
If $\Expect[T_\epp] < \infty$, and because $\epp$ was chosen arbitrarily, we have
\[
\Expect[\eta] \leq \frac{b}{q} \Big(1+o(1)\Big) \text{ as } b\to \infty.
\]
This also motivates the need for conditions on finiteness of higher order moments of $T_\epp$ to
obtain upper bound on the moments of the detection delay.
\end{enumerate}
\end{proof}

From the above theorem, the following corollary easily follows.
\begin{corrly}[\cite{tart-veer-siamtpa-2005}]\label{eq:ShiryaevnnoniddCorollary}
 If the likelihood ratios are such that \eqref{eq:noniidCond_exist_q},  \eqref{eq:noniidCond_nmoment_exists} and \eqref{eq:noniidCond_Avg_moment_exists} are satisfied for some $r \geq 1$, then
for the Shiryaev stopping time $\taus$ defined in \eqref{eq:Shiryaevalgo}
 \begin{equation}
\inf_{\tau\in \Cc_\alpha} ~ \Expect[(\tau-\Gamma)^+]^m \sim \Expect[(\taus-\Gamma)^+]^m
\sim \left(\frac{|\log\alpha|}{q + d}\right)^m \text{ as } \alpha \to 0, \text{ for all $m \leq r$.}
\end{equation}
\end{corrly}
A similar result can be concluded for the conditional moments as well.

\subsection{Performance analysis for i.i.d. model with geometric prior}
\label{sec:PerfAnaBayesIID}
We now present the second order asymptotic analysis of the Shiryaev algorithm for the i.i.d. model, provided in \cite{tart-veer-siamtpa-2005}
using the tools from nonlinear renewal theory introduced in Section~\ref{sec:renewalTheory}.

When the observations $\{X_n\}$ are i.i.d. conditioned on the change point, condition \eqref{eq:noniidCond_exist_q} is satisfied
and
\begin{equation*}
\frac{1}{t}\sum_{i=n}^{k+t} Y_i \to D(f_1\|  f_0) \ \ \ \text{ a.s. }\Prob_n \ \ \text{ when } t \to \infty  \ \ \forall n,
\end{equation*}
where $D(f_1\| f_0)$ is the K-L divergence between the densities $f_1$ and $f_0$ (see Definition~\ref{def:K-L-Divergence}).
From Thereom~\ref{thm:BayesianAsympVVV}, it follows that for $a_\alpha = \frac{1-\alpha}{\alpha}$,
\[\PFA(\taus) \leq \alpha.\]
Also, it is shown in \cite{tart-veer-siamtpa-2005} that if
\[0 < D(f_1 \| f_0) < \infty \ \ \ \ \text{ and } \ \ \  0 < D(f_0 \| f_1) < \infty,\]
then conditions \eqref{eq:noniidCond_nmoment_exists} and \eqref{eq:noniidCond_Avg_moment_exists} are satisfied, and hence
from Corollary~\ref{eq:ShiryaevnnoniddCorollary}
\begin{equation}
\label{eq:PerfShiryaev}
\inf_{\tau\in \Cc_\alpha} ~ \Expect[(\tau-\Gamma)^+]^m \sim \Expect[(\taus-\Gamma)^+]^m
\sim \left(\frac{|\log\alpha|}{D(f_1\|  f_0) + |\log(1-\rho)|}\right)^m \text{ as } \alpha \to 0.
\end{equation}
Note that the above statement provides the asymptotic delay performance of the
Shiryaev algorithm.
However, the bound for $\PFA$ can be quite loose and the first order expression for the $\ADD$ \eqref{eq:PerfShiryaev}
may not provide good estimate for $\ADD$ if the $\PFA$ is not very small.
In that case it is useful to have a second order estimate based on nonlinear renewal theory,
as obtained in \cite{tart-veer-siamtpa-2005}.

First note that \eqref{eq:Lambda_n} for the i.i.d. model will reduce to
\begin{eqnarray}
\label{eq:iidLambda_n}
\Lambda_{n+1} = \frac{1 + \Lambda_n}{1-\rho} L (X_{n+1}).
\end{eqnarray}
with $\Lambda_0=0$.
Now, let $Z_n = \log \Lambda_n$, and recall that
$Y_k = \log \frac{f_1(X_{k})}{f_0(X_{k})}$. Then, it can be shown that
\begin{eqnarray}
\label{eq:ZnEqSnEtan}
Z_n &=& \sum_{k=1}^n [Y_k + |\log (1-\rho)|] + \log\left(e^{Z_0} + \rho\right) + \sum_{k=1}^{n-1} \log\left(1 + e^{-Z_k}\rho\right) \nonumber\\
&=& \sum_{k=1}^n Y_k + n |\log (1-\rho)| + \eta_n.
\end{eqnarray}
Therefore the Shiryaev algorithm can be equivalently written as
\[\taus = \inf\{n\geq 1: Z_n \geq b\}.\]

We now show how the asymptotic overshoot distribution plays a key role in second order asymptotic analysis of $\PFA$ and $\ADD$.
Since, $p_{\taus} \geq A$ implies that $Z_{\taus}\geq \log \frac{A}{1-A} = \log a = b$, we have,
\[\frac{1}{1+e^{-Z_{\taus}}} \geq \frac{a}{1+a}.\]
\begin{eqnarray*}
\PFA(\taus) = \Expect[1-p_{\taus}] &= &\Expect\left[\frac{1}{1+e^{Z_{\taus}}}\right] \leq \Expect\left[e^{-Z_{\taus}}\right].\\
&=& \Expect\left[\frac{1}{e^{Z_{\taus}}} \frac{1}{1+e^{-Z_{\taus}}}\right]
\geq \Expect\left[\frac{1}{e^{Z_{\taus}}} \frac{a}{1+a}\right] = \Expect\left[e^{-Z_{\taus}}\right](1+o(1)) \text{ as } a \to \infty.
\end{eqnarray*}
Thus,
\begin{eqnarray*}
\PFA(\taus) &=& \Expect[e^{-Z_{{\taus}}}](1+o(1)) = e^{-b} \Expect[e^{-(Z_{{\taus}} - b)}](1+o(1)) \text{ as } b \ \to \infty.\\
&=& e^{-b} \Expect[e^{-(Z_{{\taus}} - b)}|  \tau \geq \Gamma](1+o(1)) \ \ \ \ \ \  \text{ as } b \ \to \infty.
\end{eqnarray*}
and we see that $\PFA$ is a function of the overshoot when $Z_n$ crosses $a$ from below.

Similarly,
\[Z_{\taus} = \sum_{k=1}^{{\taus}} Y_k + {\taus} |\log (1-\rho)| + \eta_{\taus}.\]
Following the developments in Section~\ref{sec:renewalTheory}, if the sequence $Y_n$ satisfies some additional
conditions, then we can write \footnote{As explained in \cite{tart-veer-siamtpa-2005}, this analysis is facilitated if we restrict to the worst-case detection delay, which is obtained by conditioning on the event that the change happens at time 1.}:
\begin{eqnarray*}
\Expect_1 [\tau]  &=& \frac{\Expect_1 [Z_\tau] -\Expect_1 [\eta_\tau]}{|\log (1-\rho)| + D(f_1\| f_0)} \\
               &=& \frac{b + \Expect_1 [Z_\tau-b] - \Expect_1 [\eta_\tau]}{|\log (1-\rho)| + D(f_1\| f_0)}.
\end{eqnarray*}

It is shown in \cite{tart-veer-siamtpa-2005} that $\eta_n$ is a slowly changing sequence,
and hence the distribution of $Z_{\taus}- b$, as $b\to \infty$,
is equal to the asymptotic distribution of the overshoot
when the random walk $\sum_{k=1}^n Y_k + n |\log (1-\rho)|$ crosses a large positive boundary.
We define the following quantities: the asymptotic overshoot distribution of a random walk
\begin{equation}\label{eq:ODistRx}
R(x) = \lim_{b\to \infty} \Prob\left(\sum_{k=1}^{\taus} Y_k + {\taus} |\log (1-\rho)| - b \leq x\right)
\end{equation}
its mean
\begin{equation}\label{eq:ODistMean}
\kappa = \int_0^\infty x dR(x)
\end{equation}
and its Laplace transform at 1
\begin{equation}\label{eq:ODistLaplace}
\zeta = \int_0^\infty e^{-x} dR(x).
\end{equation}

Also, the sequence $\{Y_n\}$ satisfies some additional conditions and hence the following results are true.
\begin{theorem}[\cite{tart-veer-siamtpa-2005}] \label{thm:sec_order}
If $\{Y_n\}$s are nonarithmetic then $\eta_n$ is a slowly changing sequence. Then by Theorem \ref{thm:NonlinearRenewalTheort}
\[
\lim_{b\to \infty} \Prob(Z_{\taus}-b \leq x) = R(x).
\]
This implies
\[
\PFA(\taus) \sim  \zeta \ e^{-b} \quad  \text{ as } b \ \to \infty.
\]
If in addition $\Expect[Y^2]<\infty$ then
\[
\Expect_1[\taus] = \frac{b + \kappa - \Expect_1[\eta]}{|\log (1-\rho)| + D(f_1\| f_0)} + o(1), \text{ as } b \ \to \infty.
\]
where $\eta$ is the a.s. limit of the sequence $\{\eta_n\}$.
\end{theorem}

In Table~\ref{Tab:ShiryaevPerf}, we compare the asymptotic expressions for $\PFA$ and $\ADD$ given in Theorem~\ref{thm:sec_order} with simulations. As can be seen in the table, the asymptotic expressions get more accurate as $\PFA$ approaches $0$.
\begin{table} [htbp]
 \centering
 \centering
 {
 \caption{$f_0 \sim {\cal N}(0,1)$, $f_1 \sim {\cal N}(1,1)$, $\rho=0.01$}
 \label{Tab:ShiryaevPerf}
\begin{tabular}{|l|l|l|l|l|l|l|l|l|l|l|l|l|}
\hline
&\multicolumn{2}{c|}{$\PFA$}&\multicolumn{2}{c|}{$\ADD$ }\\
\cline{1-5}
$b$ &   Simulations & Analysis     &   Simulations &   Analysis  \\
    &               & Theorem 3.3  &               &   Theorem 3.3 \\
\hline
1.386 & 1.22 $\times 10^{-1}$& 1.39 $\times 10^{-1}$& 6.93 & 10.31 \\
2.197 & 5.85 $\times 10^{-2}$& 6.19 $\times 10^{-2}$& 8.87 & 11.9 \\
4.595 & 5.61 $\times 10^{-3}$& 5.63 $\times 10^{-3}$& 13.9 & 16.6 \\
6.906 & 5.59 $\times 10^{-4}$& 5.58 $\times 10^{-4}$& 18.59 & 21.13 \\
11.512 & 5.6 $\times 10^{-6}$& 5.58 $\times 10^{-6}$& 27.64 & 30.16 \\
\hline
\end{tabular}}
 \end{table}

In Fig.~\ref{fig:Shiryaev_Tradeoff_slope} we plot the  $\ADD$ as a function of $\log(\PFA)$ for Gaussian observations. For a $\PFA$ constraint of $\alpha$ that is small, $b\approx |\log(\alpha)|$, and
\[
\ADD \approx \Expect_1[\taus] \approx \frac{|\log(\alpha)|}{|\log (1-\rho)| + D(f_1\| f_0)}
\]
giving a slope of $\frac{1}{|\log (1-\rho)| + D(f_1\| f_0)}$ to the trade-off curve.

When $|\log (1-\rho)| \ll D(f_1\| f_0)$, the observations contain more information about the change than the prior, and the tradeoff slope is roughly $\frac{1}{D(f_1\| f_0)}$. On the other hand,
when $|\log (1-\rho)| \gg D(f_1\| f_0)$,  the prior  contains more information about the change than the observations, and the tradeoff slope is roughly $\frac{1}{|\log (1-\rho)|}$. The latter asymptotic slope
is achieved by the stopping time that is based only on the prior:
\[
\tau = \inf\left\{ n\geq 1: \Prob (\Gamma > n)  \leq \alpha \right\}.
\]
This is also easy to see from \eqref{eq:recursionTake}. With $D(f_1\| f_0)$ small, $L(X)\approx 1$, and the recursion
for $p_k$ reduces to
\[
p_{n+1}=p_n + (1-p_n)\rho, \quad  p_0=0.
\]
Expanding we get $p_n = \rho \sum_{k=0}^{n-1} (1-\rho)^k$. The desired expression for the mean delay is obtained from the equation
$p_{\tau} = 1-\alpha$.
\begin{figure}[!t]
\center
\includegraphics[width=11cm, height=7cm]{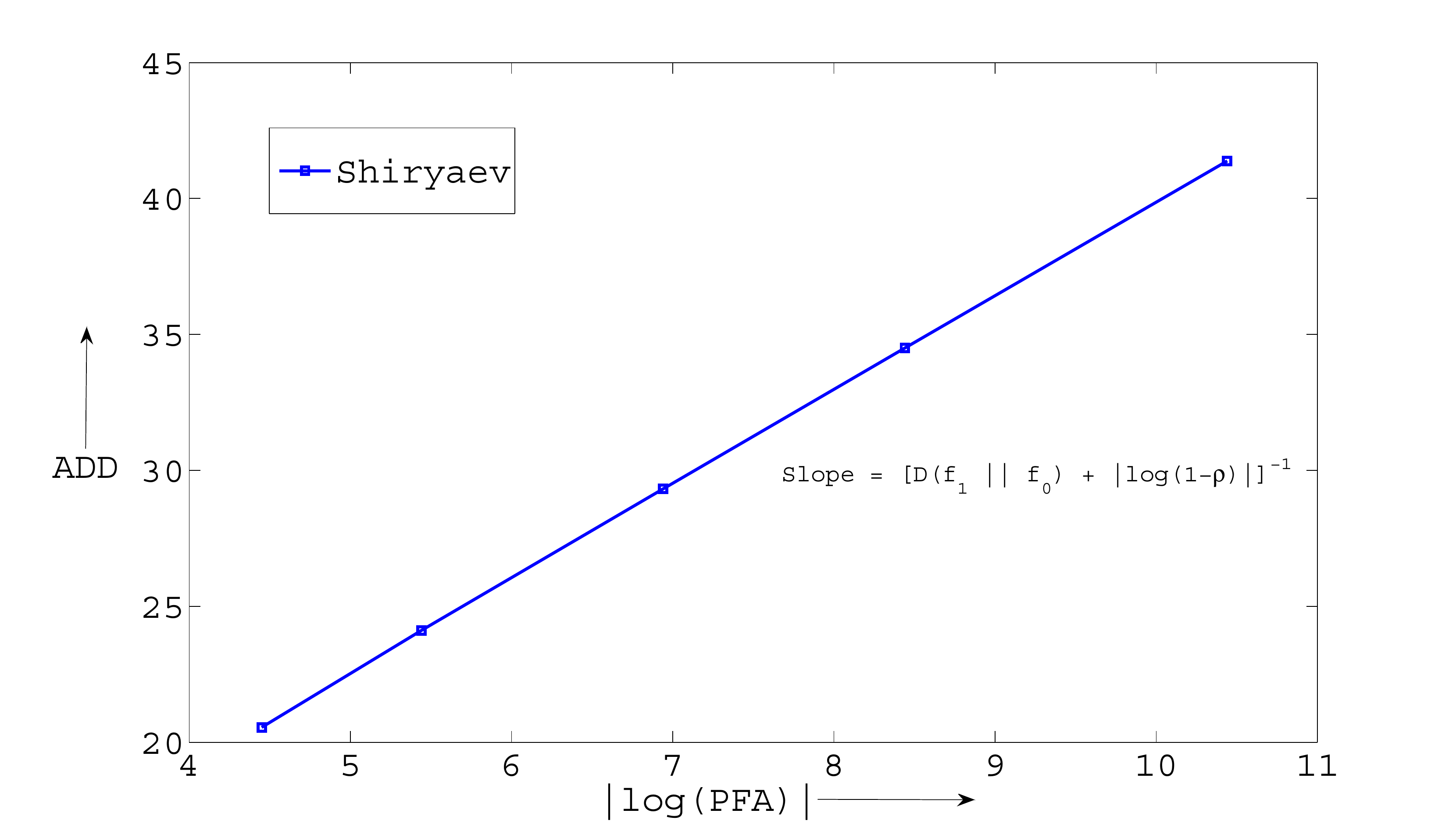}
\caption{$\ADD-\PFA$ trade-off curve for the Shiryaev algorithm: $\rho=0.01$, $f_0=\mathcal{N}(0,1)$, $f_1=\mathcal{N}(0.75,1)$}
\label{fig:Shiryaev_Tradeoff_slope}
\end{figure}

\section{Minimax quickest change detection}
\label{sec:MinimaxCent}
When the distribution of the change point is not known, we may model the change point as a deterministic but unknown positive integer $\gamma$. A number of heuristic algorithms have been developed in this setting. The earliest work is due to Shewhart \cite{shew-jamstaa-1925,shew-book-1931}, in which the log likelihood based on the current observation is compared with a threshold to make a decision about the change. The motivation for such a technique is based on the following fact:  if $X$ represents the generic random variable for the i.i.d. model with $f_0$ and $f_1$ as the pre- and post-change p.d.fs, then
\begin{equation}\label{eq:KLNumbers}
\Expect_\infty \left[\log L(X) \right]  = - D(f_0 \| f_1) < 0  \quad \text{and } \quad  \Expect_1\left[\log L(X) \right] =  D(f_1 \| f_0)  > 0.
\end{equation}
where as defined earlier $L(x) = f_1(x)/f_0(x)$, and $\Expect_\infty$ and $\Expect_1$ correspond to expectations when $\gamma=\infty$ and $\gamma=1$, respectively.
Thus,  after $\gamma$, the log likelihood of the observation $X$ is more likely to be above a given threshold.
Shewhart's method is widely employed in practice due to its simplicity; however, significant performance gain can be achieved by making use of past observations to make the decision about the change. Page \cite{page-biometrica-1954} proposed such an algorithm that uses past observations, which he called the CuSum algorithm.
The motivation for the CuSum algorithm is also based on \eqref{eq:KLNumbers}.
By the law of large numbers,
$\sum_{i=\gamma}^n \log L (X_i)$ grows to $\infty$ as $n \to \infty$.
Thus, if $S_n = \sum_{i=1}^n \log L (X_i)$ is the accumulated log likelihood sum, then before $\gamma$,
$S_n$ has a \emph{negative drift} and evolves towards $-\infty$. After $\gamma$,
$S_n$ has a \emph{positive drift} and climbs towards $\infty$. Therefore, intuition suggests the following algorithm that detects this change in drift, i.e.,
\begin{equation}\label{eq:PageCuSum_driftbased}
 \tauc = \inf\left\{n\geq 1: \left(S_n - \min_{1\leq k \leq n } S_k\right)   \geq b \right\}.
\end{equation}
where $b >0$.
Note that
\[
S_n - \min_{1\leq k \leq n } S_k
= \max_{0 \leq k \leq n }\sum_{i=k+1}^n \log L (X_i)
= \max_{1\leq k \leq n+1 }\sum_{i=k}^n \log L (X_i).
\]
Thus, $\tauc$ can be equivalently defined as follows.
\begin{algorithm}[CuSum algorithm]\label{algo:CuSum}
\begin{equation}\label{eq:CuSum}
\tauc = \inf\left\{n\geq 1: W_n \geq b \right\}.
\end{equation}
where
\begin{equation} \label{eq:Wn}
W_n = \max_{1\leq k \leq n+1 }\sum_{i=k}^n \log  L (X_i)
\end{equation}
\end{algorithm}
The statistic $W_n$ has the convenient recursion:
\begin{equation}\label{eq:CuSum_recursion}
W_{n+1} = {\left(W_n + \log L (X_{n+1})\right)}^+ , \quad W_0=0.
\end{equation}
It is this cumulative sum recursion that led Page to call $\tauc$ the CuSum algorithm.

The summation on the right hand side (RHS) of \eqref{eq:Wn} is assumed to take the value 0 when $k=n+1$. It turns out that one can get an algorithm that is equivalent to the above CuSum algorithm by removing the term $k=n+1$ in the maximization on the RHS of \eqref{eq:Wn}, to get the statistic:
\begin{equation} \label{eq:Cn}
C_n = \max_{1\leq k \leq n }\sum_{i=k}^n \log L (X_i)
\end{equation}
The statistic $C_n$ also has a convenient recursion:
\begin{equation}\label{eq:CuSum_recursion2}
C_{n+1} = \left(C_n \right)^+  + \log L (X_{n+1}) , \quad C_0=0,
\end{equation}
Note that unlike the statistic $W_n$, the statistic $C_n$ can be negative. Nevertheless it is easy to see that both $W_n$ and $C_n$ cross will cross a positive threshold $b$ at the same time (sample path wise) and hence the CuSum algorithm can be equivalently defined in terms of $C_n$ as:
\begin{equation}\label{eq:CuSumCn}
\tauc = \inf\left\{n\geq 1: C_n \geq b \right\}.
\end{equation}
An alternative way to derive the CuSum algorithm is through the maximum likelihood approach
i.e., to compare the likelihood of $\{\Gamma \leq n\}$ against $\{\Gamma > n\}$. Formally,
\begin{equation}\label{eq:PageCuSum_MLBased}
 \tauc = \inf\left\{n\geq 1: \frac{\max_{1\leq k \leq n } \prod_{i=1}^{k-1} f_0(X_i) \prod_{i=k}^n f_1(X_i)}{\prod_{i=1}^{n} f_0(X_i)} \geq B \right\}.
\end{equation}
Cancelling terms and taking log in \eqref{eq:PageCuSum_MLBased} gives us \eqref{eq:CuSumCn} with $b= \log B$.

See Fig.~\ref{fig:CuSum_ADD_PFA_Evolution} for a typical evolution of the CuSum algorithm.

Although, the CuSum algorithm was developed heuristically by Page \cite{page-biometrica-1954}, it was later shown in
\cite{lord-amstat-1971}, \cite{mous-astat-1986}, \cite{ritov-astat-1990} and \cite{lai-ieeetit-1998}, that it has
very strong optimality properties. In this section, we will study the CuSum and related algorithms from a fundamental viewpoint, and discuss how each of these algorithms is provably optimal with respect to a meaningful and useful optimization criterion.
\begin{figure}[!t]
\center
\includegraphics[width=11cm, height=7cm]{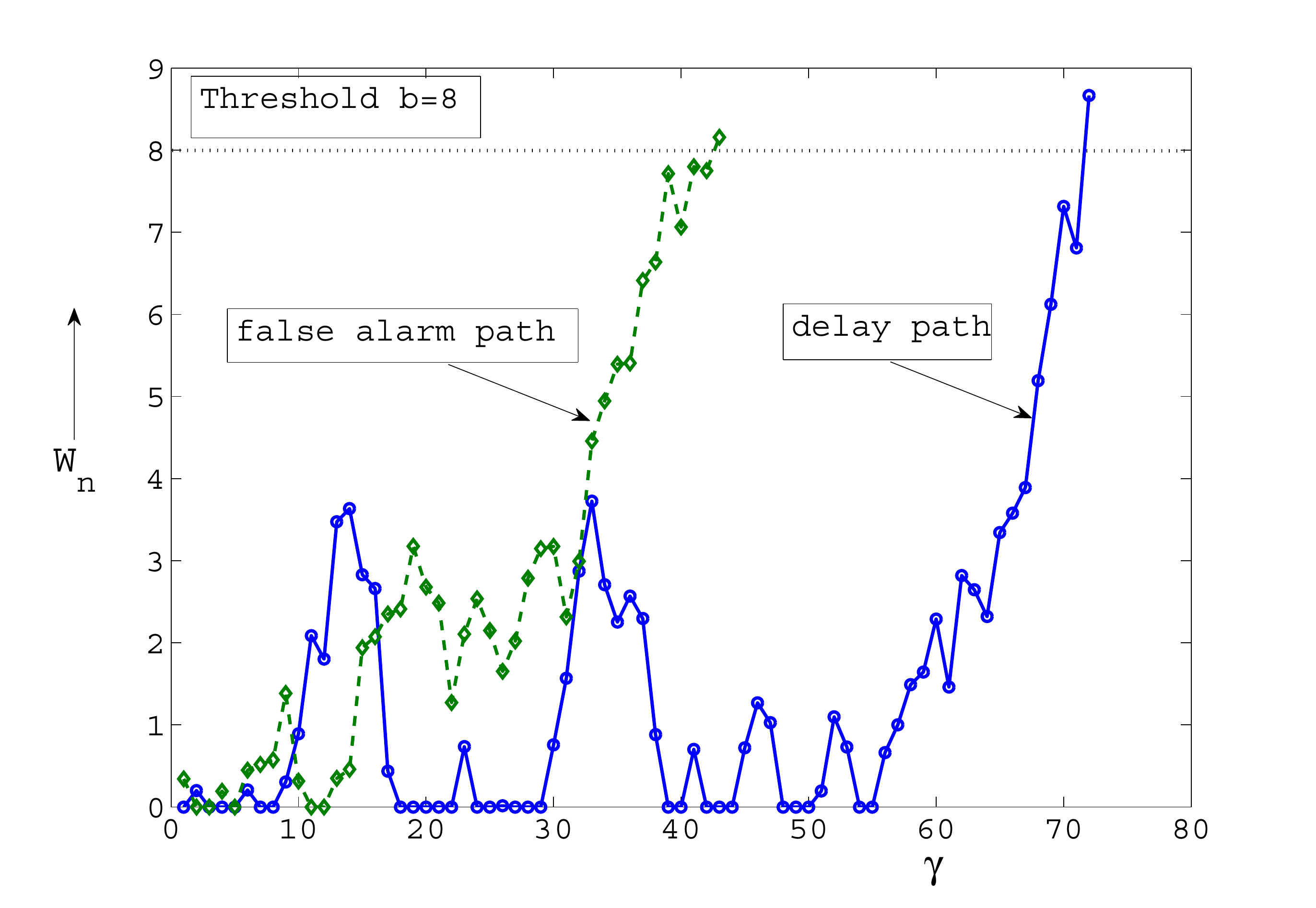}
\caption{Typical Evolution of the CuSum algorithm. Threshold $b=8$ and change point $\gamma=60$. }
\label{fig:CuSum_ADD_PFA_Evolution}
\end{figure}

Without a prior on the change point, a reasonable measure of false alarms is the mean time to false alarm, or its reciprocal, which is the false alarm rate ($\FAR$):
\begin{equation}\label{eq:FARDef}
 \FAR(\tau) = \frac{1}{\Expect_\infty[\tau]}.
\end{equation}
Finding a uniformly powerful test that minimizes the delay over all possible values of $\gamma$ subject to
a $\FAR$ constraint is generally not possible. Therefore it is more appropriate to study the quickest  change detection problem in a minimax setting in this case. There are two important minimax problem formulations, one  due to Lorden \cite{lord-amstat-1971} and the other due to Pollak \cite{poll-astat-1985}.

In Lorden's formulation, the objective is to minimize the supremum of the average delay conditioned on the worst possible realizations, subject to a constraint on the false alarm rate. In particular, if we define\footnote{Lorden defined $\WADD$ with $(\tau -n +1)^{+}$ inside the expectation, i.e., he assumed a delay penalty of 1 if the algorithm stops at the change point. We drop this additional penalty in our definition in order to be consistent with the other delay definitions in this chapter.}
\begin{equation}\label{eq:WADDdef}
\WADD(\tau) = \underset{n \geq 1}{\operatorname{\sup}}\ \text{ess sup} \ \Expect_n\left[(\tau-n)^+| X_1, \dots, X_{n-1}\right].
\end{equation}
and denote the set of stopping times that meet a constraint $\alpha$ on the $\FAR$ by
\begin{equation}\label{eq:DalphaDef}
\Dd_{\alpha}=\{\tau: \FAR(\tau) \leq \alpha\}
\end{equation}
we have the following problem formulation due to Lorden.
\begin{equation}
\label{eq:LordenProb2}
\text{\emph{Lorden's Problem:} For a given $\alpha$, find a stopping time $\tau \in \Dd_\alpha$  to minimize $\WADD (\tau)$.}
\end{equation}

%

For the i.i.d. setting, Lorden showed that the CuSum algorithm \eqref{eq:CuSum} is asymptotically optimal for Lorden's formulation \eqref{eq:LordenProb2} of as $\alpha\to 0$. It was later shown in  \cite{mous-astat-1986} and \cite{ritov-astat-1990}  that the CuSum algorithm is actually exactly optimal
for \eqref{eq:LordenProb2}. Although the CuSum algorithm enjoys such a strong optimality property under Lorden's formulation, it can be argued that $\WADD$ is a somewhat pessimistic measure of delay. A less pessimistic way to measure the delay was suggested by Pollak \cite{poll-astat-1985}:
\begin{equation}\label{eq:PollakCADDDef}
\CADD(\tau) = \underset{n \geq 1}{\operatorname{\sup}}\ \Expect_n[\tau-n| \tau\geq n].
\end{equation}
for all stopping times $\tau$ for which the expectation is well-defined.
\begin{lemma} \label{lemma:CADD_WADD}
\[
\WADD(\tau) \geq \CADD(\tau).
\]
\end{lemma}
\begin{proof} Due to the fact that $\tau$ is a stopping time on $\{X_n\}$,
\[
\{\tau \geq n\} = \{\tau \leq n-1 \}^{c} \in \sigma (X_1, X_2, \ldots, X_{n-1})
\]
Therefore, for each $n$
\[
\text{ess sup} \ \Expect_n\left[(\tau-n)^+| X_1, \dots, X_{n-1}\right] \geq  \Expect_n[(\tau-n)^{+} | \tau\geq n] = \Expect_n[\tau-n| \tau\geq n]
\]
and the lemma follows.
\end{proof}

We now state Pollak's formulation of the problem that uses $\CADD$ as the measure of delay.
\begin{equation}
\label{eq:PollakProb3}
\text{\emph{Pollak's Problem:} For a given $\alpha$, find a stopping time $\tau \in \Dd_\alpha$  to minimize $\CADD (\tau)$.}
\end{equation}
Pollak's formulation has been studied in the i.i.d.\ setting in \cite{poll-astat-1985} and
\cite{tart-poll-polu-arxiv-2011}, where it is shown that algorithms based on the Shiryaev-Roberts statistic (to be defined later) are within a constant of the best possible performance over the class $\Dd_{\alpha}$, as $\alpha\to 0$.

Lai \cite{lai-ieeetit-1998}  studied both \eqref{eq:LordenProb2} and \eqref{eq:PollakProb3}  in a non-i.i.d. setting and
developed a general minimax asymptotic theory for these problems. In particular, Lai obtained a lower bound on $\CADD(\tau)$, and hence also on the $\WADD(\tau)$, for every stopping time in the class $\Dd_{\alpha}$, and showed that an extension of the CuSum algorithm for the non-i.i.d. setting achieves this lower bound asymptotically as $\alpha\to 0$.

In Section~\ref{sec:minimaxAlgo} we introduce a number of alternatives to the CuSum algorithm for minimax quickest change detection in
the i.i.d. setting that are based on the Bayesian Shiryaev algorithm. We then discuss the optimality properties of these algorithms in Section~\ref{sec:MinimaxOpt}.
While we do not discuss the exact optimality of the CuSum algorithm from \cite{mous-astat-1986} or \cite{ritov-astat-1990}, we briefly discuss the asymptotic optimality result from \cite{lord-amstat-1971}.
We also note that the asymptotic optimality of the CuSum algorithm for both \eqref{eq:LordenProb2} and \eqref{eq:PollakProb3} follows from the results in the non-i.i.d. setting of \cite{lai-ieeetit-1998}, which are summarized in Section~\ref{sec:LAI}.

\subsection{Minimax Algorithms Based on the Shiryaev Algorithm}
\label{sec:minimaxAlgo}

Recall that the Shiryaev algorithm is given by  (see \eqref{eq:ShiryaevRnrhoExpanded} and \eqref{eq:ShiryaevRnrho}):
\[
\taus = \inf\left\{ n\geq 1: R_{n,\rho} \geq a \right\}
\]
where
\[
R_{n,\rho} = \frac{1}{(1-\rho)^n} \sum_{k=1}^n (1-\rho)^{k-1} \prod_{i=k}^{n} L(X_i).
\]
Also recall that $R_{n,\rho}$ has the recursion:
\[
R_{n+1,\rho} = \frac{1+R_{n,\rho}}{1-\rho} L(X_{n+1}), \quad R_{0,\rho} = 0.
\]
Setting $\rho=0$ in the expression for  $R_{n,\rho}$ we get the Shiryaev-Roberts (SR)
statistic \cite{robe-technometrics-1966}:
\begin{equation} \label{eq:SRstat}
R_{n} = \sum_{k=1}^n \prod_{i=k}^{n} L(X_i)
\end{equation}
with the recursion:
\begin{equation} \label{eq:SRrec}
R_{n+1} = (1+R_n) L(X_{n+1}), \quad R_0 = 0.
\end{equation}

\begin{algorithm}[Shiryaev-Roberts $\mathrm{(SR)}$ Algorithm] \label{algo:SR}
\begin{equation}
\label{eq:SRalgo}
\tausr = \inf\left\{ n\geq 1: R_n \geq B, \ \ R_0=0 \right\}.
\end{equation}
\end{algorithm}
It is shown in \cite{poll-astat-1985} that the SR algorithm is the limit of a sequence of Bayes tests, and in that limit it is asymptotically Bayes efficient. Also, it is shown in \cite{tart-poll-polu-arxiv-2011} that the SR algorithm is {\em second order} asymptotically optimal for \eqref{eq:PollakProb3}, as $\alpha \to 0$, i.e., its delay is within a constant of the best possible delay over the class $\Dd_{\alpha}$. Further, in \cite{poll-tart-statsinica-2009}, the SR algorithm is shown to be exactly optimal with respect to a number of other interesting criteria.

It is also shown in \cite{poll-astat-1985} that a modified version of the SR algorithm, called the Shiryaev-Roberts-Pollak (SRP) algorithm, is {\em third order} asymptotically optimal for \eqref{eq:PollakProb3}, i.e., its delay is within a constant of the best possible delay over the class $\Dd_{\alpha}$, and the constant goes to zero as $\alpha \to 0$. To introduce the SRP algorithm, let $Q^{B}$ be the quasi-stationary distribution of the SR statistic $R_n$ above:
\[
\Prob(R_0 \leq x) = Q^{B}(x) = \lim_{n\to \infty} \Prob_0(R_n \leq x | \tausr>n).
\]
The new recursion, called the Shiryaev-Roberts-Pollak (SRP) recursion, is given by,
\[
R_{n+1}^B = (1+R_n^B) L(X_{n+1})
 \]
with $R_0^B$ is distributed according to $Q^{B}$.
\begin{algorithm}[Shiryaev-Roberts-Pollak $\mathrm{(SRP)}$ Algorithm]
\begin{equation}
\label{eq:SRPalgo}
\tausrp = \inf\left\{ n\geq 1: R_n^B \geq B \ \   \right\}.
\end{equation}
\end{algorithm}
Although the SRP algorithm is strongly asymptotically optimal for Pollak's formulation of \eqref{eq:PollakProb3}, in practice, it is difficult to compute the quasi-stationary distribution $Q^{B}$. A numerical framework for computing  $Q^{B}$ efficiently is provided in \cite{mous-polu-tart-statsin}. Interestingly, the following modification of the SR algorithm with a \emph{specifically
designed} starting point $R_0=r \geq 0$ is found to outperform the SRP procedure uniformly over all possible values of the change point \cite{tart-poll-polu-arxiv-2011}. This modification, referred to as the SR-$r$ procedure, has
the recursion:
\[
R_{n+1}^r = (1+R_n^r) L(X_{n+1}), \quad R_0 = r.
\]
\begin{algorithm}[Shiryaev-Roberts-$r$ (SR-$r$) Algorithm]
\begin{equation}
\label{eq:SRralgo}
\tausrr = \inf\left\{ n\geq 1: R_n^r \geq B \right\}.
\end{equation}
\end{algorithm}
It is shown in \cite{tart-poll-polu-arxiv-2011} that the SR-$r$ algorithm is also third order  asymptotically optimal for \eqref{eq:PollakProb3}, i.e., its delay is within a constant of the best possible delay over the class $\Dd_{\alpha}$, and the constant goes to zero as $\alpha \to 0$.
%

Note that for an arbitrary stopping time, computing the $\CADD$ metric \eqref{eq:PollakCADDDef} involves taking supremum over all possible change times, and computing the $\WADD$ metric \eqref{eq:WADDdef} involves another supremum over all possible past realizations of observations. While we can analyze the performance of the proposed algorithms through bounds and asymptotic approximations (as we will see in Section~\ref{sec:MinimaxOpt}, it is not obvious how one might evaluate $\CADD$ and $\WADD$ for a given algorithm in computer simulations. This is in contrast with the Bayesian setting, where $\ADD$ (see \eqref{eq:ADD_def}) can easily be evaluated in simulations by averaging over realizations of change point random variable $\Gamma$.

Fortunately, for the CuSum algorithm \eqref{eq:CuSum} and for the Shiryaev-Roberts algorithm \eqref{eq:SRalgo}, both $\CADD$ and $\WADD$ are easy to evaluate in simulations due to the following lemma.
\begin{lemma} \label{lemma:CADDeqWADD}
\begin{equation} \label{eq:CADDeqWADDCuSum}
\CADD (\tauc) = \WADD (\tauc) = \Expect_1 \left[ (\tauc - 1)\right].
\end{equation}
\begin{equation} \label{eq:CADDeqWADDSR}
\CADD (\tausr) = \WADD (\tausr) = \Expect_1 \left[ (\tausr - 1)\right].
\end{equation}
\end{lemma}
\begin{proof}
The CuSum statistic $W_n$ (see \eqref{eq:Wn}) has initial value $0$ and  remains non-negative for all $n$. If the change were to happen at some time $n > 1$,
then the pre-change statistic $W_{n-1}$ is greater than or equal  $0$, which equals the pre-change statistic if the change happens at $n=1$. Therefore, the delay for the CuSum statistic to cross a positive threshold $b$ is largest when the change happens at $n=1$, irrespective of the realizations of the observations, $X_1, X_2, \ldots, X_{n-1}$. Therefore
\[
\WADD (\tauc) = \underset{n \geq 1}{\operatorname{\sup}}\ \text{ess sup} \ \Expect_n \left[(\tauc-n)^+ | X_1, \dots, X_{n-1}\right] = \Expect_1 \left[ (\tauc - 1)^{+} \right] = \Expect_1 \left[ (\tauc - 1) \right] .
\]
and
\[
\CADD(\tauc) = \underset{n \geq 1}{\operatorname{\sup}}\ \Expect_n[\tau-n| \tauc \geq n] =  \Expect_1 \left[ (\tauc- 1) | \tauc \geq 1\right] = \Expect_1 \left[ (\tauc- 1) \right].
\]
This proves \eqref{eq:CADDeqWADDCuSum}. A similar argument can be used to establish \eqref{eq:CADDeqWADDSR}.
\end{proof}
Note that the above proof crucially depended on the fact that both the CuSum algorithm and the Shiryaev-Roberts algorithm start with the initial value of $0$.
Thus it is not difficult to see that Lemma~\ref{lemma:CADDeqWADD} does not hold for the SR-$r$ algorithm, unless of course $r=0$.
Lemma~\ref{lemma:CADDeqWADD} holds partially for the SRP algorithm since the initial distribution $Q^B$ makes the statistic $R_n^B$ stationary in $n$.
As a result $\Expect_n[\tausrp-n| \tausrp \geq n]$ is the same for every $n$.
However, as mentioned previously, $Q^B$ is difficult to compute in practice, and this makes the evaluation of $\CADD$ and $\WADD$ in simulations somewhat challenging.
\subsection{Optimality Properties of the Minimax Algorithms}
\label{sec:MinimaxOpt}
In this section we first show that the algorithms based on the Shiryaev-Roberts statistics, SR, SRP, and SR-$r$ are asymptotically optimal for Pollak's formulation of \eqref{eq:PollakProb3}. We need an important theorem that is proved in \cite{poll-tart-statsinica-2009}.
\begin{theorem}[\cite{poll-tart-statsinica-2009}]
\label{thm:SRLB}
If the threshold in the SR algorithm \eqref{eq:SRalgo} can be selected to meet the constraint $\alpha$ on $\FAR$ with equality, then
\[
\tausr =
 \underset{\tau\in \Dd_{\alpha}}{\operatorname{\arg \min}}
\frac{\sum_{n=1}^\infty \Expect_n[(\tau-n)^+]}{\Expect_\infty[\tau]}
\]
\end{theorem}
\begin{proof}
We give a sketch of the proof here. Note that
\begin{eqnarray*}
\sum_{n=1}^\infty \Expect_n[(\tau-n)^+] &=& \sum_{n=1}^\infty  \Expect_n[\tau-n| \tau \geq n] \Prob_\infty(\tau \geq n)\\
                                          &\leq & \CADD(\tau)  \sum_{n=1}^\infty \Prob_\infty(\tau \geq n)\\
                                          &=& \CADD(\tau) \Expect_\infty[\tau],
\end{eqnarray*}
and hence is well defined for all stopping times for which $\CADD$ and $\FAR$ exist.
The first part of the proof is to show that
\[
\sum_{n=1}^\infty \Expect_n[(\tausr-n)^+] = \underset{\tau\in \Dd_{\alpha}}{\operatorname{\min}} \sum_{n=1}^\infty \Expect_n[(\tau-n)^+].\]
This follows from the following result of \cite{poll-astat-1985}. If
\[
J_{\lambda,\rho}(\tau) = \min_{\tau} \Big(\Expect\left[(\tau - \Gamma)^+\right] + \lambda \; \Prob(\tau<\Gamma)\Big)\]
with $\Gamma$ having the geometric distribution of \eqref{eq:geom_prior} with parameter $\rho$, and
\begin{equation} \label{eq:lambdarho}
\tau_{\lambda,\rho} = \underset{\tau}{\operatorname{\arg \min}} J_{\lambda,\rho}(\tau).
\end{equation}
Then for a given $\tausr$ (with a given threshold),
there exists a sequence $\{(\lambda_i, \rho_i)\}$ and with $\lambda_i \to \lambda^*$ and $\rho_i \to 0$ such that $\tau_{\lambda_i,\rho_i}$ converge to $\tausr$, as $i \to \infty$.
Thus, the SR algorithm is the limit of a sequence of Bayes tests. Moreover,
\[
\underset{\rho \to 0, \lambda \to \lambda^*}{\operatorname{\lim \sup}}
\frac{1-J_{\lambda,\rho}(\tau_{\lambda,\rho})}{1-J_{\lambda,\rho}(\tausr)} = 1.
\]
By \eqref{eq:lambdarho}, for any stopping time $\tau$,  it holds that
\[
\frac{1-J_{\lambda,\rho}(\tau)}{1-J_{\lambda,\rho}(\tausr)} \leq \frac{1-J_{\lambda,\rho}(\tau_{\lambda,\rho})}{1-J_{\lambda,\rho}(\tausr)}.
\]
Now by taking the limit $\rho \to 0, \lambda \to \lambda^*$ on both sides, using the fact that
for any stopping time $\tau$ \cite{poll-tart-statsinica-2009}
\[
\frac{1-J_{\lambda,\rho}(\tau)}{\rho} \to \lambda^* \Expect_\infty[\tau] - \sum_{n=1}^\infty \Expect[(\tau-n)^+] \ \ \text{ as } \  \rho \to 0
\]
and using the hypothesis of the theorem that the $\FAR$ constraint can be met with equality by using $\tausr$, we have the desired result.

The next step in the proof is to show that it is enough to consider stopping times in the
class $\Dd_{\alpha}$ that meet the constraint of $\alpha$ with quality.
The result then follows easily from the fact that $\tausr$ is optimal with respect
to the numerator in $\frac{\sum_{n=1}^\infty \Expect_n[(\tau-n)^+]}{\Expect_\infty[\tau]}$.
\end{proof}

We now state the optimality proof for the procedures SR, SR-$r$ and SRP. We only provide an outline of the proof to illustrate the fundamental ideas behind the result.
\begin{theorem}[\cite{tart-poll-polu-arxiv-2011}]\label{thm:SRFamilyPerfOpt}
If $\Expect_1[\log\frac{f_1(X_n)}{f_0(X_n)}]^2 < \infty$, and $\log\frac{f_1(X_n)}{f_0(X_n)}$ is nonarithmetic then
\begin{enumerate}
\item \begin{equation} \label{eq:SRThirdOrderLB}
\underset{\tau\in \Dd_{\alpha}}{\operatorname{\inf}}\ \CADD(\tau) \geq \frac{|\log \alpha|}{D(f_1||f_0)} + \hat{\kappa} + o(1) \text{ as } \alpha\to 0
\end{equation}
where $\hat{\kappa}$ is a constant that can be characterized using renewal theory \cite{wood-nonlin-ren-th-book-1982}.
\item For the choice of threshold $B=\frac{1}{\alpha}$, $\FAR(\tausr)\leq \alpha$, and
\begin{equation}\label{eq:SRThirdOrder}
\begin{split}
\CADD(\tausr) =& \frac{|\log \alpha|}{D(f_1||f_0)} + \hat{\zeta} + o(1) \text{ as } \alpha\to 0 \\
=& \underset{\tau\in \Dd_{\alpha}}{\operatorname{\inf}}\ \CADD(\tau) + O(1) \text{ as } \alpha\to 0
\end{split}\end{equation}
where $\hat{\zeta}$ is again a constant that can be characterized using renewal theory \cite{wood-nonlin-ren-th-book-1982}.
\item There exists a choice for the threshold $B=B_\alpha$ such that $\FAR(\tausrp)\leq \alpha(1+o(1))$ and
\begin{equation}\label{eq:SRPThirdOrder}
\begin{split}
\CADD(\tausrp) =& \frac{|\log \alpha|}{D(f_1||f_0)} + \hat{\kappa} + o(1) \text{ as } \alpha\to 0 \\
=& \underset{\tau\in \Dd_{\alpha}}{\operatorname{\inf}}\ \CADD(\tau) + o(1) \text{ as } \alpha\to 0.
\end{split}\end{equation}
\item There exists a choice for the threshold $B=B_\alpha$ such that $\FAR(\tausrp)\leq \alpha(1+o(1))$ and a choice for
the initial point $r$ such that
\begin{equation}\label{eq:SRRThirdOrder}
\begin{split}
\CADD(\tausrr) =& \frac{|\log \alpha|}{D(f_1||f_0)} + \hat{\kappa} + o(1) \text{ as } \alpha\to 0 \\
=& \underset{\tau\in \Dd_{\alpha}}{\operatorname{\inf}}\ \CADD(\tau) + o(1) \text{ as } \alpha\to 0.
\end{split}\end{equation}
\end{enumerate}
\end{theorem}
\begin{proof}
To prove that the above mentioned choice of thresholds meets the $\FAR$ constraint, we note that $ \{R_n^r-n-r\} $ is  a martingale. Specifically, $ \{R_n-n\} $ is a martingale and the conditions of theorem \ref{thm:DoobOptSamp} are satisfied  \cite{poll-astat-1987}. Hence,
\[
\Expect_\infty[R_{\tausr}-\tausr]=0.
\]
Since, $\Expect_\infty[R_{\tausr}] \geq B$, for $B=\frac{1}{\alpha}$ we have
\[
\FAR(\tausr) = \frac{1}{\Expect_\infty[R_{\tausr}]} \leq \frac{1}{B} = \alpha.
\]
For a description of how to set the thresholds for $\tausrr$ and $\tausrp$, we refer the reader to
\cite{tart-poll-polu-arxiv-2011}.

The proofs of the delay expressions for all the algorithms have a common theme. The first part of these proofs is based on Theorem \ref{thm:SRLB}. We first show that $\frac{\sum_{n=1}^\infty \Expect_n[(\tau-n)^+]}{\Expect_\infty[\tau]}$ is a lower bound to $\CADD(\tau)$.
\begin{eqnarray*}
\CADD(\tau) = \sup_n \Expect_n[\tau-n| \tau\geq n] &=& \frac{\sum_{j=1}^\infty \sup_n \Expect_n[\tau-n| \tau\geq n] \Prob_\infty(\tau\geq j)}{\Expect_\infty[\tau]}\\
                                   &\geq & \frac{\sum_{j=1}^\infty \Expect_j[\tau-j| \tau\geq j] \Prob_\infty(\tau\geq j)}{\Expect_\infty[\tau]}\\
                                   &=& \frac{\sum_{j=1}^\infty \Expect_j[(\tau-j)^+]}{\Expect_\infty[\tau]}.
\end{eqnarray*}
From Theorem \ref{thm:SRLB}, since $\tausr$ is optimal with respect to minimizing $\frac{\sum_{n=1}^\infty \Expect_n[(\tau-n)^+]}{\Expect_\infty[\tau]}$,
we have
\[
\underset{\tau\in \Dd_{\alpha}}{\operatorname{\inf}}\ \CADD(\tau) \geq \frac{\sum_{n=1}^\infty \Expect_n[(\tausr-n)^+]}{\Expect_\infty[\tausr]}.
\]
The next step is to use nonlinear renewal theory (see Section~\ref{sec:renewalTheory}) to obtain a second order approximation for the right hand side of the above equation, as we did  for the Shiryaev procedure in Section~\ref{sec:PerfAnaBayesIID}:
\[
\frac{\sum_{n=1}^\infty \Expect_n[(\tausr-n)^+]}{\Expect_\infty[\tausr]} = \frac{|\log \alpha|}{D(f_1||f_0)} + \hat{\kappa} + o(1) \text{ as } \alpha\to 0.
\]
The final step is to show that the $\CADD$ for the SR-$r$ and SRP procedures are within $o(1)$, and the $\CADD$ for SR procedure is within $O(1)$ of this lower bound \eqref{eq:SRThirdOrderLB}. This is done by obtaining second order approximations using nonlinear renewal theory for the $\CADD$ of SR, SRP, SR-$r$ procedures, and get \eqref{eq:SRThirdOrder}, \eqref{eq:SRPThirdOrder} and \eqref{eq:SRRThirdOrder}, respectively.
\end{proof}

It is shown in \cite{poll-tart-statsinica-2009} that $\frac{\sum_{n=1}^\infty \Expect_n[(\tau-n)^+]}{\Expect_\infty[\tau]}$ is also equivalent to the average delay when the change happens at a ``far horizon":  $\gamma\to\infty$. Thus, the SR algorithm is also optimal with respect to that criterion.

The following corollary follows easily from the above two theorems. Recall that in the minimax setting,
an algorithm is called third order asymptotically
optimum if its delay is within an $o(1)$ term of the best possible, as the $\FAR$ goes to zero.
An algorithm is called second order asymptotically
optimum if its delay is within an $O(1)$ term of the best possible, as the $\FAR$ goes to zero.
And an algorithm is called first order asymptotically optimal
if the ratio of its delay with the best possible goes to 1, as the $\FAR$ goes to zero.
\begin{corrly}
Under the conditions of Theorem \ref{thm:SRLB}, the SR-$r$ \eqref{eq:SRralgo} and the SRP \eqref{eq:SRPalgo} algorithms are third order asymptotically optimum, and
the SR algorithm is second order asymptotically optimum for the Pollak's formulation \eqref{eq:PollakProb3}. Furthermore, using the choice of thresholds specified in Theorem \ref{thm:SRLB} to meet the $\FAR$ constraint of $\alpha$, the asymptotic performance of all three algorithms are equal up to first order:
\[
\CADD (\tausr) \sim \CADD (\tausrp) \sim \CADD(\tausrr) \sim \frac{|\log \alpha|}{D(f_1 || f_0)}.
\]
\end{corrly}
Furthermore, by Lemma~\ref{lemma:CADDeqWADD}, we also have:
\[
\WADD (\tausr) \sim  \frac{|\log \alpha|}{D(f_1 || f_0)}.
\]

In \cite{lord-amstat-1971} the asymptotic optimality of the CuSum algorithm \eqref{eq:CuSum} as $\alpha\to 0$ is established for Lorden's formulation of \eqref{eq:LordenProb2}.
First, ergodic theory is used to show that choosing the threshold $b=|\log \alpha|$ ensures $\FAR(\tauc) \leq \alpha$.
For the above choice of threshold $B=|\log \alpha|$, it is shown that
\[\WADD(\tauc) \leq \frac{|\log \alpha|}{D(f_1\| f_0)}(1 + o(1)) \text{ as } \alpha\to 0.\]
Then the exact optimality of the SPRT \cite{wald-wolf-amstat-1948} is used to find a lower bound on the $\WADD$
of the class $\Dd_{\alpha}$:
\[
\underset{\tau\in \Dd_{\alpha}}{\operatorname{\inf}}\ \WADD(\tau) \geq \frac{|\log \alpha|}{D(f_1\| f_0)} (1+o(1))\text{ as } \alpha\to 0.
\]
These arguments establish the first order asymptotic optimality of the CuSum algorithm  for Lorden's formulation.
Futhermore, as we will see later in Theorem~\ref{thm:Lai1998Proof}, Lai \cite{lai-ieeetit-1998} showed that:
\[
\underset{\tau\in \Dd_{\alpha}}{\operatorname{\inf}}\ \WADD(\tau)
     \geq \underset{\tau\in \Dd_{\alpha}}{\operatorname{\inf}}\ \CADD(\tau) \geq \frac{|\log \alpha|}{I}\Big(1+o(1)\Big).
 \]
Thus by Lemma~\ref{lemma:CADDeqWADD}, the first order asymptotic optimality of the CuSum algorithm extends to Pollak's formulation \eqref{eq:PollakProb3}, and we have the following result.
\begin{corrly}
The CuSum algorithm \eqref{eq:CuSum} with threshold $b=|\log \alpha|$ is first order asymptotically optimum for both Lorden's and Pollak's formulations. Furthermore,
\[
\CADD (\tauc) = \WADD (\tauc) \sim  \frac{|\log \alpha|}{D(f_1 || f_0)}.
\]
\end{corrly}

In Fig.~\ref{fig:CuSum_Tradeoff_slope}, we plot the trade-off curve for the CuSum algorithm, i.e., plot $\CADD$
as a function of $-\log \FAR$. Note that the curve has a slope of $1/D(f_1 || f_0)$.
\begin{figure}[!t]
\center
\includegraphics[width=11cm, height=7cm]{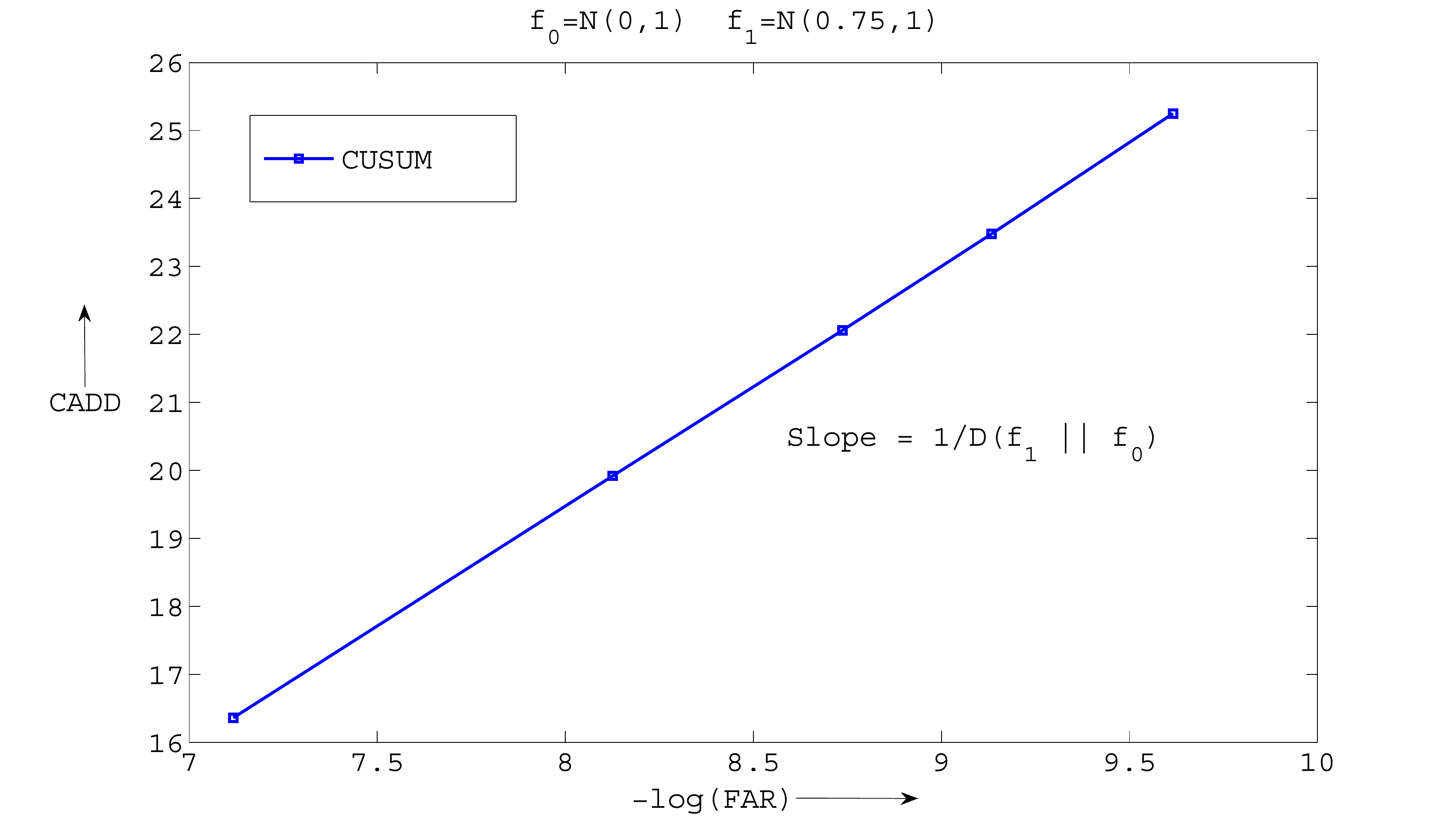}
\caption{$\ADD-\PFA$ trade-off curve for the CuSum algorithm: $f_0=\mathcal{N}(0,1)$, $f_1=\mathcal{N}(0.75,1)$}
\label{fig:CuSum_Tradeoff_slope}
\end{figure}

\subsection{General Asymptotic Minimax Theory}
\label{sec:LAI}

In \cite{lai-ieeetit-1998}, the non-i.i.d. setting earlier discussed in Section~\ref{sec:NONIIDBayes} is considered, and
asymptotic lower bounds on the $\WADD$ and $\CADD$ for stopping times in $\Dd_{\alpha}$ are obtained under quite general conditions. It is then shown that the an extension of the CuSum algorithm \eqref{eq:CuSum} to this non-i.i.d. setting achieves this lower bound asymptotically as $\alpha\to 0$.

Recall the non-i.i.d.\ model from Section~\ref{sec:NONIIDBayes}. When the process evolves in the pre-change regime, the conditional density of $X_n$ given $X_1^{n-1}$ is $f_{0,n}(X_n|X_1^{n-1})$. After the change happens, the conditional density of $X_n$ given $X_1^{n-1}$  is given by $f_{1,n}(X_n|X_1^{n-1})$. Also
\[
Y_i = \log \frac{f_{1,i}(X_i|X_1^{i-1})}{f_{0,i}(X_i|X_1^{i-1})}.
\]
The CuSum algorithm can be generalized to the non-i.i.d.\ setting as follows:
\begin{algorithm}[Generalized CuSum algorithm]\label{algo:GeneralCuSum}
Let
\[
C_n = \max_{1\leq k\leq n} \sum_{i=k}^n Y_i.
\]
Then the stopping time for the generalized CuSum is
\begin{equation}
\label{eq:CuSumgen}\taug = \inf\left\{ n\geq 1: C_n \geq b \right\}
\end{equation}
\end{algorithm}

Then the following result is proved in \cite{lai-ieeetit-1998}.
\begin{theorem}\label{thm:Lai1998Proof}
If there exists a positive constant $I$ such that the $\{Y_i\}$s satisfy the following condition
\begin{equation}\label{LaiZiCond1}
\lim_{t\to \infty} \underset{n \geq 1}{\operatorname{\sup}} \ \underset{}{\operatorname{ess \sup}}
\ \Prob_n\left( \underset{m\leq t }{\operatorname{\max}} \sum_{i=n}^{n+m} Y_i \geq I (1+\delta) n \Big| X_1, \ldots, X_{n-1}\right) = 0 \ \ \ \ \forall \delta>0
\end{equation}
then, as $\alpha\to 0$
\begin{equation} \label{eq:Lai_lb}
\underset{\tau\in \Dd_{\alpha}}{\operatorname{\inf}}\ \WADD(\tau)
     \geq \underset{\tau\in \Dd_{\alpha}}{\operatorname{\inf}}\ \CADD(\tau) \geq \frac{|\log \alpha|}{I}\Big(1+o(1)\Big).
\end{equation}
Further
\[\Expect_\infty[\taug] \geq e^b, \]
and under the additional condition of
\begin{equation}\label{LaiZiCond2}
\lim_{t\to \infty} \underset{m \geq n}{\operatorname{\sup}} \ \underset{}{\operatorname{ess \sup}}
\ \Prob_n\left( t^{-1} \sum_{i=m}^{m+t} Y_i \geq I-\delta \Big| X_1, \ldots, X_{m-1}\right) = 0 \ \ \ \ \forall \delta>0,
\end{equation}
we have
\[\CADD(\taug) \leq \WADD(\taug) \leq \frac{b}{I} \Big(1+o(1)\Big) \text{ as } b \to \infty.\]
Thus, if we set $b=|\log \alpha|$, then
\[
\FAR (\taug) = \frac{1}{\Expect_\infty[\taug]} \leq \alpha
\]
and
\[
\WADD(\taug) \leq \frac{|\log \alpha|}{I}(1+o(1))
\]
which is asymptotically equal to the lower bound in \eqref{eq:Lai_lb} up to first order. Thus $\taug$ is first-order asymptotically optimum within the class $\Dd_{\alpha}$ of tests that meet the $\FAR$ constraint of $\alpha$.
\end{theorem}
\begin{proof}
We only provide a sketch of the proof for the lower bound since it also based on the idea of using Chebyshev's inequality.
The fundamental idea of the proof is to use Chebyshev's inequality to get a lower bound
on any arbitrary stopping time $\tau$ from $\Dd_{\alpha}$,
such that the lower bound is not a function of $\tau$. The lower bound obtained
is then a lower bound on the $\CADD$ for the entire family $\Dd_{\alpha}$.

Let $L_\alpha$ and $V_\alpha$ be positive constants.
To show that
\[
\underset{n \geq 1}{\operatorname{\sup}}\ \Expect_n[\tau-n| \tau\geq n] \geq L_\alpha(1+o(1)) \text{ as } \alpha \to 0
\]
it is enough to show that there exists $n$ such that
\[
\Expect_n[\tau-n| \tau\geq n] \geq L_\alpha(1+o(1)) \text{ as } \alpha \to 0.
\]
This $n$ is obtained from the following condition. Let $m$ be any positive integer. Then
if $\tau \in \Dd_{\alpha}$,  there exists $n$ such that
\begin{equation}\label{eq:thmLaiCond}
\Prob_\infty(\tau \geq n) > 0    \quad \text{ and } \quad \Prob_\infty(\tau < n + m | \tau \geq  k) \leq m \alpha.
\end{equation}
We use the $n$ that satisfies the condition \eqref{eq:thmLaiCond}.
Then, by Chebyshev's inequality
\[
\Prob_n(\tau-n \geq L_\alpha | \tau \geq n ) \leq (L_\alpha)^{-1} \Expect_n[\tau-n| \tau \geq n].
\]
We can then write
\[
\Expect_n[\tau-n| \tau \geq n] \geq L_\alpha \ \Prob_n(\tau-n \geq L_\alpha | \tau \geq n ).
\]
Now it has to be shown that $\Prob_n(\tau-n \geq L_\alpha | \tau \geq n ) \to 1$ uniformly over the family $\Dd_{\alpha}$. To show this, we condition on $V_\alpha$.
\begin{eqnarray*}
\Prob_n\left(\tau-n < L_\alpha | \tau \geq n \right) &=& \Prob_n\left(\tau-n < L_\alpha ; \sum_{i=n}^\tau Y_i < V_\alpha | \tau \geq n \right) \\
                                          && + \Prob_n\left(\tau-n < L_\alpha ; \sum_{i=n}^\tau Y_i \geq  V_\alpha | \tau \geq n \right)
\end{eqnarray*}
The trick now is to use the hypothesis of the theorem and choose proper values of $V_\alpha$ and $L_\alpha$ such that the two terms on the right hand side of the above equations are bounded by terms that go to zero and are not a function of the stopping time $\tau$. We can write
\begin{eqnarray*}
\Prob_n\left(\tau-n < L_\alpha ; \sum_{i=n}^\tau Y_i \geq  V_\alpha | \tau \geq n \right) \leq
\underset{}{\operatorname{ess \sup}}\  \Prob_n \left( \underset{t\leq L_\alpha}{\operatorname{\max}} \sum_{i=n}^{n+t} Y_i \geq V_\alpha | X_1, \ldots, X_{n-1} \right).
\end{eqnarray*}
In \cite{lai-ieeetit-1998}, it is shown that if we choose
\[
L_\alpha = (1-\delta) \frac{|\log \alpha|}{I}
\]
and
\[
V_\alpha = (1-\delta^2) |\log \alpha|
\]
then the condition \eqref{LaiZiCond1} ensures that the above probability goes to zero.
The other term goes to zero by using a change of measure argument and using condition \eqref{eq:thmLaiCond}:
\[
\Prob_n\left(\tau-n < L_\alpha ; \sum_{i=n}^\tau Y_i < V_\alpha | \tau \geq n \right) \leq e^{V_\alpha} \Prob_\infty(\tau < n+ L_\alpha | \tau \geq n).
\]
\end{proof}

\section{Relationship Between The Models}
We have discussed the Bayesian formulation of the quickest change detection problem in Section~\ref{sec:BayesCent} and the minimax formulations
of the problem in Section~\ref{sec:MinimaxCent}.
The choice between the Bayesian and the minimax formulations is obvious, and is governed by the knowledge of the distribution of $\Gamma$. However,
it is not obvious which of the two minimax formulations should be chosen for a given application.
As noted earlier, the formulations of Lorden and Pollak are equivalent for low $\FAR$ constraints, but differences arise for
moderate values of $\FAR$ constraints.
%
%
Recent work by Moustakides \cite{mous-astat-2008}
has connected these three formulations and found possible interpretations for each of them.
We summarize the result below.

Consider a model in which
the change point is dependent on the stochastic process. That is, the probability that change happens at time $n$ depends on $\{X_1, \ldots, X_n\}$. Let
\[\pi_n = \Prob(\Gamma=n | X_1, \ldots, X_n).\]
The distribution of $\Gamma$ thus belongs to a family of distributions.
%
%
Assume that while we are trying to find a suitable
stopping time $\tau$ to minimize delay, an adversary is searching for a process $\{\pi_n\}$ such that the delay for any stopping time is maximized.
That is the adversary is trying to solve
\[J(\tau) = \underset{\{\pi_n\}}{\operatorname{\sup}}\ \ \Expect[\tau-\Gamma| \tau\geq \Gamma].\]
It is shown in \cite{mous-astat-2008} that if the adversary has access to the observation sequence, and uses this information to choose $\pi_n$, then $J(\tau)$ becomes the delay expression in Lorden's formulation \eqref{eq:LordenProb2} for a given $\tau$, i.e.,
\[J(\tau) = \underset{n \geq 1}{\operatorname{\sup}}\ \text{ess sup} \ \Expect_n[(\tau-n)^+| X_1, \dots, X_{n-1}].\]
However, if we assume that the adversary does not have access to the observations, but only has access to the test performance, then $J(\tau)$ is equal to the delay in Pollak's formulation \eqref{eq:PollakProb3}, i.e.,
\[J(\tau) = \underset{n \geq 1}{\operatorname{\sup}}\ \Expect_n[\tau-n| \tau\geq n].\]
Finally, the delay for the Shiryaev formulation \eqref{eq:BayesConstProb} corresponds to the case when $\pi_n$ is restricted to only one possibility, the distribution of $\Gamma$.


\section{Variants and Generalizations of the Quickest Change Detection Problem}
\label{sec:QCDVariants}
In the previous sections we reviewed the state-of-the-art for quickest change detection in a single sequence of random variables, under the assumption of complete knowledge of the pre- and post-change distributions.
%
%
%
In this section we review three important variants and extensions of the classical quickest change detection problem, where significant progress has been made. We discuss other variants of the change detection problem as part of our future research section.

\subsection{Quickest Change Detection with unknown pre- or post-change distributions}
\label{sec:GLRT}
In the previous sections we discussed quickest change detection problem when both the pre- and post-change distributions are completely specified. Although this assumption is a bit restrictive, it helped in obtaining
recursive algorithms with strong optimality properties in the i.i.d. setting, and allowed
the development of asymptotic optimality theory in a very general non-i.i.d. setting.
In this section,  we provide a review of the existing results for the case where  this assumption on the knowledge of the pre- and post-change distributions is relaxed.

Often it is reasonable to assume that the pre-change distribution is known. This is the case when changes
occur rarely and/or the decision maker has opportunities to estimate the pre-change distribution parameters
before the change occurs. When the post-change distribution is unknown but the pre-change distribution is completely specified, the post-change uncertainty may be modeled by assuming that the post-change
distribution belongs to a parametric family $\{\Prob_\theta\}$. Approaches for designing algorithms in this setting include the generalized likelihood ratio (GLR) based approach or the mixture based approach In the GLR based approach, at any time step, all the past observations are used to first obtain a maximum likelihood estimate of the post-change parameter $\theta$, and then the post-change distribution corresponding to this estimate of $\theta$ is used to compute the CuSum, Shiryaev-Roberts or related statistics.
In the mixture based approach, a prior is assumed on the space of parameters, and the statistics (e.g., CuSum or Shiryaev-Roberts), computed as a function of the post change parameter, are then integrated over the assumed prior.

In the i.i.d setting this problem is studied in  \cite{lord-amstat-1971} for the case when the post-change p.d.f. belongs to a single parameter exponential family $\{ f_\theta\}$,
and the following generalized likelihood ratio (GLR) based algorithm is proposed:
\begin{equation}
\label{eq:LordenGLRT}
\taug = \inf\left\{ n\geq 1:\max_{1\leq k \leq n} \sup_{|\theta|\geq \epsilon_\alpha} \left[\sum_{i=k}^n\log\frac{f_\theta(X_i)}{f_0(X_i)}\right] \geq c_\alpha \right\}.
\end{equation}
If $\epsilon_\alpha = \frac{1}{|\log\alpha|}$, and $c_\alpha$ is of the order of $|\log\alpha|$, then it is shown in \cite{lord-amstat-1971} that as the
FAR constraint $\alpha\to 0$,
\[\FAR(\taug) \leq \alpha \Big(1+o(1)\Big),\]
and for each post-change parameter $\theta\neq 0$,
\[
\WADD^\theta(\taug) \sim  \frac{|\log\alpha|}{D(f_\theta, f_0)}.
\]
Here, the notation $\WADD^{\theta}[\cdot]$ is used to denote the $\WADD$ when the post-change
observations have p.d.f $f_\theta$. Note that $\frac{|\log\alpha|}{D(f_\theta, f_0)}$ is the best one can do
asymptotically for a given FAR constraint of $\alpha$, even when the exact post-change distribution is known.
Thus, this GLR scheme is uniformly asymptotically optimal with respect to the Lorden's criterion \eqref{eq:LordenProb2}.

In \cite{poll-sieg-atat-1975}, the post-change distribution is assumed to belong to an exponential family of p.d.f.s as above, but the GLR based test is replaced by a mixture based test. Specifically, a prior $G(\cdot)$ is assumed on the range of $\theta$ and the following test is proposed:
\begin{equation}
\label{eq:PollSiegMixture}
\taum = \inf\left\{ n\geq 1:\max_{1\leq k \leq n} \left[ \int \prod_{i=k}^n\frac{f_\theta(X_i)}{f_0(X_i)} \text{d}G(\theta)\right] \geq \frac{1}{\alpha} \right\}.
\end{equation}
For the above choice of threshold, it is shown that
\[\FAR(\taum) \leq \alpha,\]
and for each post-change parameter $\theta\neq 0$, if $\theta$ is in a set over which $G(\cdot)$ has a positive density, as the FAR constraint $\alpha\to 0$,
\[\WADD^\theta[\taum] \sim \frac{|\log\alpha|}{D(f_\theta, f_0)}.\]
Thus, even this mixture based test is uniformly asymptotically optimal with respect to the Lorden's criterion.

Although the GLR and mixture based tests discussed above are efficient, they do not have an equivalent recursive implementation, as the CuSum test or the Shiryaev-Roberts tests have when the post-change distribution is known. As a result, to implement the GLR or mixture based tests, we need to use the entire past information $(X_1, \cdots, X_n)$, which grows with $n$. In \cite{lai-ieeetit-1998}, asymptotically optimal sliding-window based GLR and mixtures tests are proposed, that only used a finite number of past observations. The window size has to be chosen carefully and is a function of the FAR. Moreover, the results here are valid even for the non-i.i.d. setting discussed earlier in Section~\ref{sec:LAI}. Recall the non-i.i.d. model from Section~\ref{sec:NONIIDBayes}, with the prechange conditional p.d.f. given by  $f_{0,n}(X_n|X_1^{n-1})$, and the post-change conditional p.d.f. given by $f_{1,n}(X_n|X_1^{n-1})$.
Then the generalized CuSum \eqref{eq:CuSumgen} GLR and mixture based algorithms are given, respectively, by:
\begin{equation}
\label{eq:LaiGLRT}
\htaug = \inf\left\{ n\geq 1:\max_{n-m_\alpha\leq k \leq n-m_\alpha^{'}} \sup_{\theta} \left[\sum_{i=k}^n\log\frac{f_{\theta,i}(X_i|X_1^{i-1})}{f_{0,i}(X_i|X_1^{i-1})}\right] \geq c_\alpha \right\}
\end{equation}
and
\begin{equation}
\label{eq:LaiMixture}
\htaum = \inf\left\{ n\geq 1:\max_{n-m_\alpha\leq k \leq n} \left[ \int \prod_{i=k}^n\frac{f_{\theta,i}(X_i|X_1^{i-1})}{f_{0,i}(X_i|X_1^{i-1})} \text{d}G(\theta)\right] \geq e^{c_\alpha}\right\}.
\end{equation}
It is shown that for a suitable choice of $c_\alpha$, $m_\alpha$ and $m_\alpha^{'}$, under some conditions on the likelihood ratios and on the distribution $G$,
 both of these tests satisfy the FAR constraint of $\alpha$. In particular,
\[\FAR(\htaum) \leq \alpha,\]
and as $\alpha\to0$,
 \[\FAR(\htaug) \leq \alpha\Big(1+o(1)\Big) .\]
Moreover, under some conditions on the post-change parameter $\theta$,
\[\WADD^\theta[\htaug] \sim \WADD^\theta[\htaum] \sim \frac{|\log\alpha|}{D(f_\theta, f_0)}.\]
Thus, under the conditions stated, the above window-limited tests are also asymptotically optimal.
We remark that, although finite window based tests are useful for the i.i.d. setting here, we still need to store
the entire history of observations in the non-i.i.d. setting to compute the conditional densities, unless the
dependence is finite order Markov.
See \cite{lai-jrss-1995} and \cite{lai-ieeetit-1998} for details.

To detect a change in mean of a sequence of Gaussian observations in an i.i.d. setting,
i.e., when $f_0\sim \mathcal{N}(0,1)$ and $f_1 \sim \mathcal{N}(\mu,1)$, $\mu\neq 0$, the GLR rule discussed above \eqref{eq:LaiGLRT} (with $m_\alpha=n$ and $m_\alpha^{'}=1$) reduces to
\begin{equation}
\label{eq:SiegmundGaussianGLRT}
\tau_v = \inf\left\{ n\geq 1:\max_{0\leq k < n} \left[\sum_{i=k}^n\frac{|S_n-S_n|}{\sqrt{(n-k)}}\right] \geq b\right\}.
\end{equation}
This test is studied in \cite{sieg-venk-astat-1995} and performance of the test, i.e., expressions for $\FAR(\tau_v)$ and $\Expect_1[\tau_v]$ are obtained. In \cite{lai-jrss-1995}, it is shown that a window-limited modification of the above scheme \eqref{eq:SiegmundGaussianGLRT} can also be designed.

When both pre- and post-change distributions are not known, again GLR based or mixture based tests have been studied and asymptotic properties characterized.  In \cite{lai-xing-sqa-2010}, this problem has been studied in the i.i.d setting (Bayesian and non-Bayesian), when both pre- and post-change distributions belong to an exponential family. For the Bayesian setting, the change point is assumed to be geometric and there are priors (again from an exponential family) on both the pre- and post-change parameters. GLR and mixture based tests are proposed that have asymptotic optimality properties. For a survey of other algorithms when both the pre- and post-change distributions are not known, see \cite{lai-jrss-1995}.


In \cite{unni-etal-ieeeit-2011}, rather than developing procedures that are uniformly asymptotically optimal
for each possible post-change distribution, robust tests are characterized when the pre- and post-change
distributions belong to some uncertainty classes. The objective is to find a stopping time that satisfies a given false alarm constraint (probability of false alarm or the $\FAR$ depending on the formulation) for each
possible pre-change distribution, and minimizes the detection delay (Shiryaev, Lorden or the Pollak version)
supremized over each possible pair of pre- and post-change distributions. It is shown that under some conditions on the uncertainty classes, the \emph{least favorable distribution} from each class can be obtained, and the robust test is the classical test designed according to the lease favorable distribution pair.
It is shown in \cite{unni-etal-ieeeit-2011} that although robust test is not uniformly asymptotically optimal, it can be significantly better than the GLR based test of \cite{lord-amstat-1971} for some parameter ranges and for moderate values of $\FAR$.  The robust solution also has the added advantage that it can be implemented in a simple recursive manner, while the GLR test does not admit a recursive solution in general, and may require the
solution to a complex nonconvex optimization problem at every time instant.
%
%
\subsection{Data-efficient quickest change detection}
\label{sec:DE-QCD}
In the classical problem of quickest change detection that we discussed in Section~\ref{sec:BayesCent}, a change in the distribution of a sequence of random variables has to be detected as soon as possible, subject to a constraint on the probability of false alarm. Based on the past information, the decision maker has to decide whether to stop and declare change or to continue acquiring more information. In many engineering applications of quickest change detection there is a cost associated with acquiring information or taking observations, e.g., cost associated with taking measurements, or  cost of batteries in sensor networks, etc. (see \cite{bane-veer-sqa-2012} for a detailed motivation and related references).
In \cite{bane-veer-sqa-2012}, Shiryaev's formulation (Section~\ref{sec:BayesCent}) is extended by including an additional constraint on the cost
of observations used in the detection process. The observation cost is captured through the average number of observations used before the change point, with the understanding that the cost of observations after the change point is included in the delay metric.
In order to minimize the average number of observations used before $\Gamma$, at each time instant, a decision is made on whether to use the observation in the next time step, based on all the
available information. Let $S_n \in \{0, 1\}$, with $S_n = 1$ if it is been decided to take the observation at time $n$, i.e. $X_n$ is available for decision making, and $S_n=0$ otherwise. Thus, $S_n$ is an on-off (binary) control input based on the information available up to time $k-1$, i.e.,
\[
S_n = \mu_{k-1} (I_{k-1}), \quad k = 1, 2, \ldots
\]
with $\mu$ denoting the control law and $I$ defined as:
\[
I_n = \left[ S_1, \ldots, S_n, X_1^{(S_1)}, \ldots, X_n^{(S_n)} \right].
\]
Here, $X_i^{(S_i)}$ represents $X_i$ if $S_i=1$, otherwise $X_i$ is absent from the information vector $I_n$. 


Let $\psi = \{\tau, \mu_0, \ldots, \mu_{\tau-1} \}$ represent a policy for data-efficient quickest change detection, where
$\tau$  is a stopping time on the information sequence $\{I_n\}$.
  The objective in \cite{bane-veer-sqa-2012} is to solve the following optimization problem:
\begin{eqnarray}
\label{eq:basicproblem}
\underset{\psi}{\text{minimize}} && \text{ADD}(\psi) = \mathrm{E}\left[(\tau - \Gamma)^+\right],\\
\text{subject to } \hspace{-0.15cm} &&\text{PFA}(\psi) = \mathrm{P}(\tau<\Gamma) \leq \alpha, \nonumber \\
\text{  and  } &&\text{ANO}(\psi) = \mathrm{E}\left[\sum_{k=1}^{\min(\tau,\Gamma-1)} S_n\right] \leq \beta. \nonumber
\end{eqnarray}
Here, $\ADD$, $\PFA$ and $\ANO$ stand for average detection delay, probability of false alarm and average number of observations used, respectively, and
$\alpha$ and $\beta$ are given constraints.

Define,
\[
p_n = \Prob ( \Gamma \leq k \; | \; I_n).
\]
Then, the two-threshold algorithm from \cite{bane-veer-sqa-2012} is:
\begin{algorithm}[DE-Shiryaev: $\psi(A,B)$]
\label{algo:TwoThreshold}
Start with $p_0=0$ and use the following control, with $B<A$, for $k\geq 0$:
\begin{equation}
\label{eq:TwoThresholdAlgo}
\begin{split}
               S_{k+1} & = \mu_n (p_n) =  \begin{cases}
               0 & \text{ if } p_n < B\\
               1 & \text{ if } p_n \geq B
               \end{cases}\\
               \tau &= \inf\left\{ k\geq 1: p_n > A \right\}.
               \end{split}
\end{equation}
The probability $p_n$ is updated using the following recursions:
\[
p_{k+1} = \begin{cases}
\tilde{p_n} = p_n + (1-p_n) \rho & \text{~if~} S_{k+1} = 0\\
\frac{\tilde{p_n} L(X_{k+1})}{ \tilde{p_n}  L(X_{k+1}) + (1-\tilde{p_n} )} & \text{~if~} S_{k+1} = 1
\end{cases}
\]
with
$L(X_{k+1}) = f_1(X_{k+1})/f_0(X_{k+1})$.
\end{algorithm}

With $B=0$ the DE-Shiryaev algorithm reduces to the Shiryaev algorithm.
When the DE-Shiryaev algorithm is employed, the a posteriori probability $p_n$ typically evolves as depicted in Fig.~\ref{fig:Pkevolution}.
\begin{figure}[!tb]
  \centering
  \subfloat[Evolution of $p_n$ for $f_0 \sim {\cal N}(0,1)$, $f_1 \sim {\cal N}(0.5,1)$, and  $\rho=0.01$, with thresholds $A=0.65$ and $B=0.2$.]{\label{fig:Pkevolution}\includegraphics[width=0.48\textwidth]{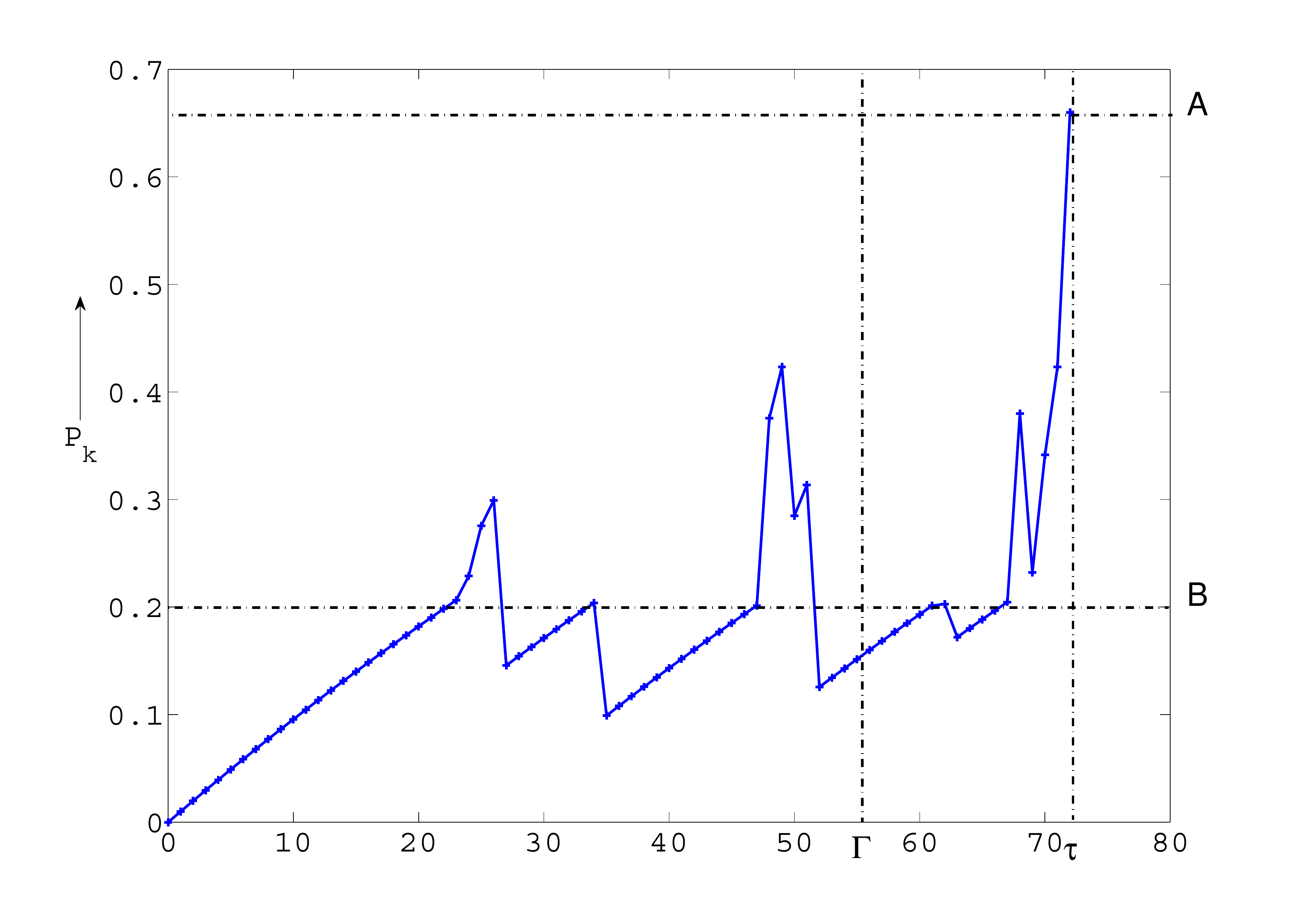}} \hspace{0.1cm}
  \subfloat[Trade-off curves comparing performance of the DE-Shiryaev  algorithm with the Fractional Sampling scheme when $B$ is chosen to achieve
  $\ANO$ equal to 50 \% of mean time to change.
  Also $f_0 \sim {\cal N}(0,1)$, $f_1 \sim {\cal N}(0.75,1)$, and $\rho=0.01$.]
  {\label{fig:CompShirDEShirFrac}\includegraphics[width=0.48\textwidth]{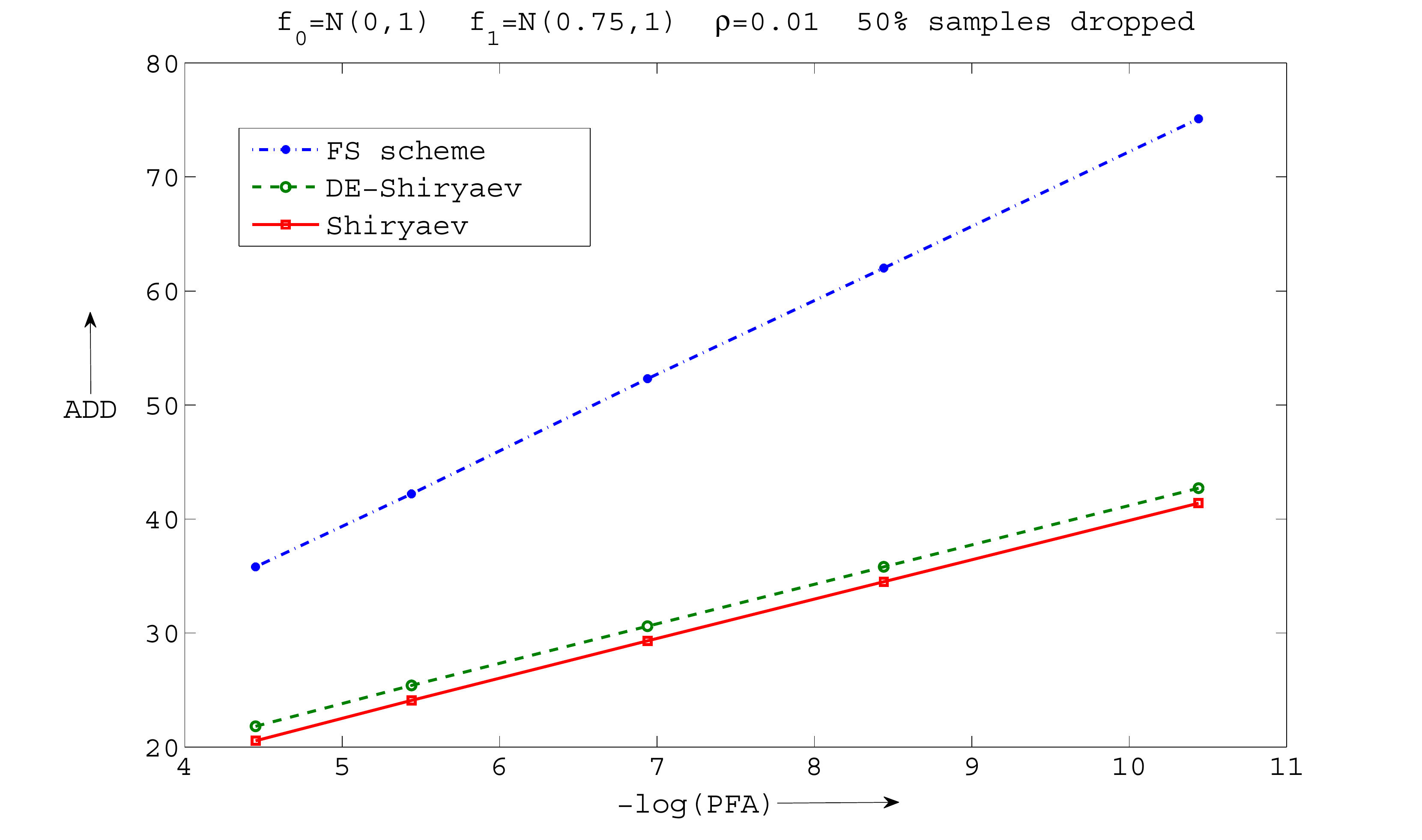}}
  \caption{Evolution and performance of the DE-Shiryaev algorithm.}
\end{figure}
It is shown in \cite{bane-veer-sqa-2012} that for a fixed $\beta$, as $\alpha\to 0$:
\begin{equation}\label{eq:BayesianDelayOpt}
\ADD(\psi(A,B)) \sim \frac{|\log(\alpha)|}{D(f_1,f_0) + |\log(1-\rho)|} \text{ as } \alpha \to 0
\end{equation}
and
\begin{equation}\label{eq:BayesianPFAOpt}
\PFA(\psi(A,B)) \sim
\alpha \left( \int_{0}^\infty e^{-x} dR(x)\right) \text{ as } \alpha \to 0.
\end{equation}
Here, $R(x)$ is the asymptotic overshoot distribution of the random walk $\sum_{k=1}^n [\log L(X_{k}) + |\log (1-\rho)|]$,
when it crosses a large positive boundary. It is shown in \cite{tart-veer-siamtpa-2005} that these are also the performance expressions for the Shiryaev
algorithm. Thus, the $\PFA$ and $\ADD$ of the DE-Shiryaev algorithm approach that of
the Shiryaev algorithm as $\alpha \to 0$, i.e., the DE-Shiryaev algorithm is asymptotically optimal.

The DE-Shiryaev  algorithm is also shown to have good delay-observation cost trade-off curves: for moderate values of probability of false alarm, for Gaussian observations, the delay of the algorithm is within 10\% of the Shiryaev delay even when the observation cost is reduced by more than 50\%.
Furthermore, the DE-Shiryaev algorithm is substantially better than the standard approach of \textit{fractional sampling} scheme, where the Shiryaev algorithm is used and where the observations to be skipped are determined a priori in order to meet the observation constraint. (See Fig.~\ref{fig:CompShirDEShirFrac}.)

In most practical applications, prior information about the distribution of the change point is not available.
As a result, the Bayesian solution is not directly applicable.  For the classical quickest change detection problem, an algorithm for the  non-Bayesian setting was obtained by taking the geometric parameter of the prior on the change point to zero (see Section~\ref{sec:minimaxAlgo}).  Such a technique cannot be used in the data-efficient setting. This is because when an observation is skipped in the DE-Shiryaev algorithm in \cite{bane-veer-sqa-2012},
the {\em a posteriori} probability is updated using the geometric prior. In the absence of prior information about the distribution of the change point,  it is by no means obvious what the right substitute for the prior is. A useful way to capture the cost of observations in a minimax setting is also needed.

In \cite{bane-veer-icassp-2012}, the minimax formulation of \cite{poll-astat-1985} is used to propose
a minimax formulation for data-efficient quickest change detection.
We observe that in the two-threshold algorithm $\psi(A,B)$, when the change occurs at a far horizon, it is the fraction of observations
taken before change that is controlled. This insight is used to formulate a \text{duty cycle} based metric to capture the cost of taking observations before the change point.
Also, we note that the duration for which observations are not taken in the algorithm in \cite{bane-veer-sqa-2012},
is also a function of the undershoot of the {\em a posteriori} probability when it goes below the threshold $B$.
Given the fact that $\frac{p_n}{1-p_n}$ for the DE-Shiryaev algorithm, has the interpretation of the
the likelihood ratio of the hypotheses ``$H_1: \Gamma \leq n$'' and ``$H_0: \Gamma > n$'',
the undershoots essentially carry the information on the likelihood ratio.
It is shown in \cite{bane-veer-icassp-2012} that this insight  can be used to design a good test in the non-Bayesian setting.
This algorithm is called the DE-CuSum algorithm and it is shown that it inherits good properties of the DE-Shiryaev algorithm. The DE-CuSum algorithm is also asymptotically optimal in a sense similar to \eqref{eq:BayesianDelayOpt} and \eqref{eq:BayesianPFAOpt}, has good trade-off curves, and performs significantly
better than the standard approach of fractional sampling.
%

\subsection{Distributed sensor systems} \label{sec:dss}
\begin{figure}
\begin{center}
\includegraphics*[width=0.4\textwidth]{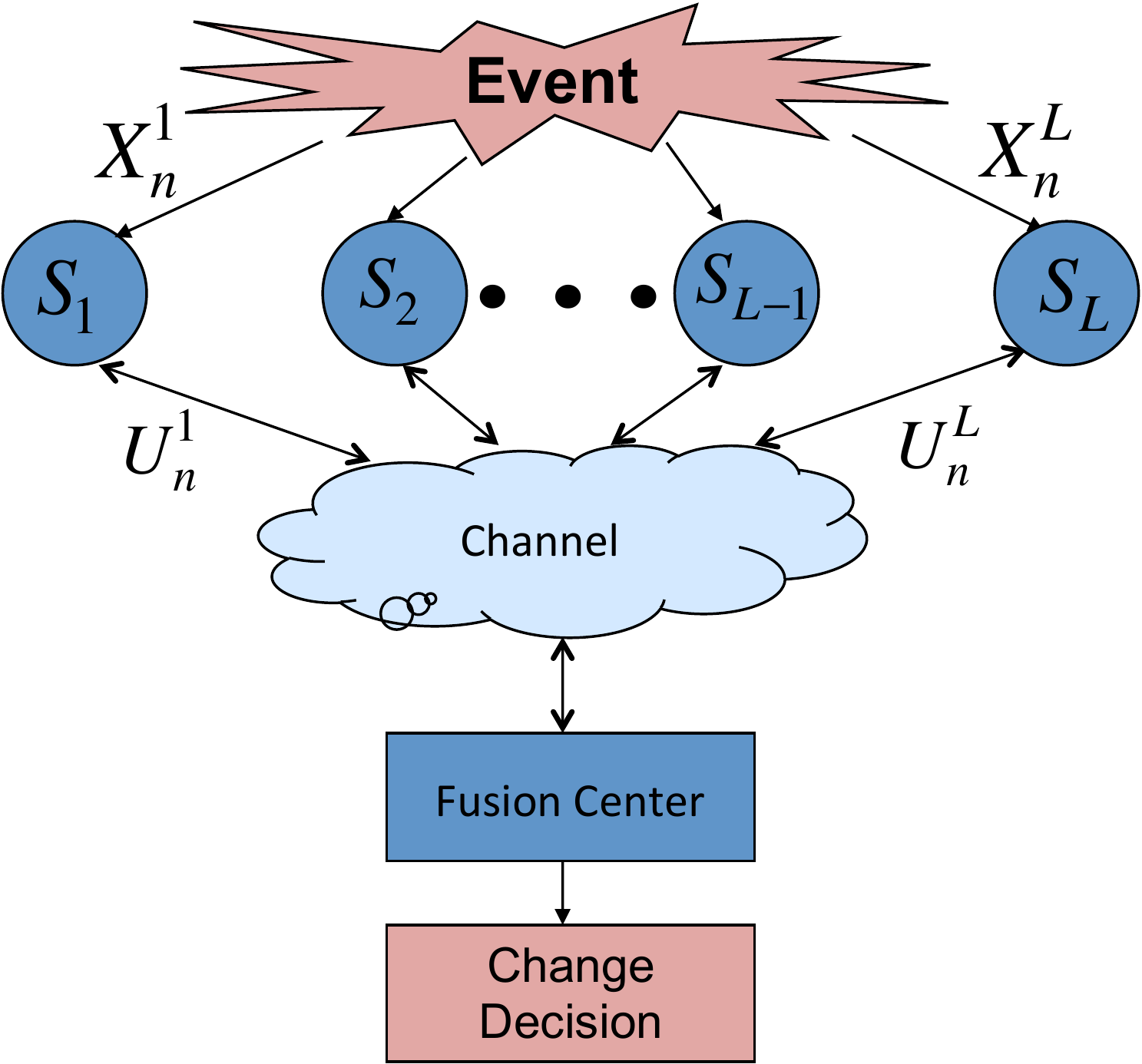}
  \caption{\small Change detection using distributed sensors
  \label{fig:ds_modes}}
\end{center}
\end{figure}
In the previous sections, we provided a  summary of existing work on quickest change detection and classification in the single sensor (equivalently, centralized multisensor) setting. For the problem of detecting biological and chemical agents, 
the setting that is more relevant is one where there is a set of distributed sensors collecting the data relevant for detection, as shown in Fig.~\ref{fig:ds_modes}. Based on the observations that the sensors receive, they send messages (which could be local decisions, but not necessarily) to a fusion center where a final decision about the hypothesis or change is made.

Since the information available for detection is distributed across the sensors in the network, these detection problems fall under the umbrella of distributed (or decentralized) detection, except in the impractical setting where all the information available at the sensors is immediately available without any errors at the fusion center.
Such decentralized decision making problems are extremely difficult. Without certain conditional independence assumptions across sensors, the problem of finding the optimal solutions, even in the simplest case of static binary hypothesis testing, is computationally intractable
\cite{tsit-dec-det-1993,vars-book-1996,will-swas-blum-ieeetsp-2000,cham-veer-ieeetsp-2003,cham-veer-ieeespm-2007,veer-vars-survey-2012}. Decentralized dynamic decision making problems, of which the quickest change detection problem is a special case, are even more challenging due to the fact that information pattern in these problems is {\em non-classical}, i.e., the different decision makers have access to different pieces of information \cite{ho-procieee-1980}.

The problem of decentralized quickest change detection  in distributed sensor systems was  introduced  in \cite{veer-ieeetit-2001}, and is described as follows. Consider the distributed multisensor system with $L$ sensors,
communicating with a fusion center shown in Fig.~\ref{fig:ds_modes}. At time $n$, an observation $X_n^\ell$ is made at sensor $\Sc_\ell$.  The changes in the statistics of the sequences $\{X_n^\ell\}$ are governed by the event. Based on the information available at time $n$, a message $U_n^\ell$ is sent from sensor $\Sc_\ell$ to the fusion center.  The fusion center may possibly feedback some control signals to the sensors based on all the messages it has received so far. For example, the fusion center might inform the sensors how to update their local decision thresholds. The final decision about the change is made at the fusion center.

There are two main approaches to generating the messages at the sensors. In the first approach, the sensors can be thought of simply quantizing their observations, i.e., $U_n^\ell$ is simply a quantized version of $X_n^\ell$. The goal then is to choose these quantizers over time and across the sensors, along with a decision rule at the fusion center, to provide the best tradeoff between detection delay and false alarms. In the second approach, the sensors perform local change detection, using all of their observations, and the fusion center combines these decisions to make the final decision about the change.

The simplest observation model for the decentralized setting is one where the sensors are affected by the change at the same time, and the observations are i.i.d. in time at each sensor and independent across sensors in both the pre-change and post-change modes. This model was introduced in \cite{veer-ieeetit-2001}, and studied in a Bayesian setting with a geometric prior on the change point for the case of quantization at the sensors.  It was shown that, unlike  the centralized problem for which the Shiryaev test is optimal (see Section~\ref{sec:BayesCent}), in the decentralized setting the optimization problem is intractable in even for this simple observation model. Some progress can be made if we allow for feedback from the fusion center \cite{veer-ieeetit-2001}. Useful results can be obtained in the asymptotic setting where the probability of false (Bayesian formulation) or false alarm rate (minimax formulation) go to zero. These results can be summarized as follows (see, e.g., \cite{tart-veer-asm-book-2004, tart-veer-siamtpa-2005,tart-veer-sqa-2008} for more details):
\begin{itemize}
\item It is asymptotically optimum for the sensors to use {\em stationary monotone likelihood ratio quantizers}, i.e., the sensors use the same quantization functions at all times, and the quantization regions are obtained by dividing the likelihood ratio of the observations into intervals and assigning the quantization levels to them in increasing order.
\item The optimum quantization thresholds at the sensors are chosen to maximize the K-L divergence between the post-change and pre-change distributions at the sensors.
\item For fixed stationary quantizers at the sensors, the fusion center is faced with a centralized quickest change detection problem. Therefore, depending on the optimization criterion (Bayes or minimax), asymptotically optimum change detection procedures can be designed using the techniques described in Sections~\ref{sec:BayesCent} and \ref{sec:MinimaxCent}
\item The tradeoff between delay and false alarms is governed by the K-L divergence of the quantized observations at the output of the sensors, and hence the first order asymptotic performance with quantization is at best equal to that without quantization.
\end{itemize}

For the case where the sensors make local decisions about the change, it is reasonable to assume that the local detection procedures use (asymptotically) optimum centralized (single sensor) statistics. For example, in the Bayesian setting, the Shiryaev statistic described in Algorithm~\ref{algo:Shiryaev} can be used, and in the minimax setting one of the statistics described in Section~\ref{sec:minimaxAlgo} can be used depending on the specific minimax criterion used. The more interesting aspect of the decision-making here is the fusion rule used to combine the individual sensor decisions. There are three main basic options that can be considered \cite{tart-veer-sqa-2008}:
\begin{itemize}
\item $\tau_\text{min}$: the fusion center stops and declares the change as soon as one of the sensors' statistics crosses its decision threshold.
\item $\tau_\text{max}$: the sensors stop taking observations once their local statistics cross their thresholds, and the fusion center stops and declares the change after all sensors have stopped.
\item $\tau_\text{all}$: the sensors continue taking observations even if their local statistics cross their thresholds, and the fusion center stops and declares the change after all the sensor statistics are above their local thresholds simultaneously.
\end{itemize}
It was first shown by \cite{mei-ieeetit-2005} that the $\tau_\text{all}$ procedure using CuSum statistics at the sensors is globally first order asymptotically optimal under Lorden's criterion \eqref{eq:LordenProb2} if the sensor thresholds are chose appropriately.  That is, the first order performance is the same as that of the centralized procedure that has access to all the sensor observations.  A more detailed analysis of minimax setting was carried out in \cite{tart-kim-fusion-2006}, in which procedures based on using CuSum and Shiryaev-Roberts statistics at the sensors were studied under Pollak's criterion \eqref{eq:PollakProb3}.  It was again shown that $\tau_\text{all}$ is globally first order asymptotically optimal, and that $\tau_\text{max}$ and $\tau_\text{min}$ are not.

For the Bayesian case, if the sensors use Shiryaev statistics, then both  $\tau_\text{max}$ and $\tau_\text{all}$  can be shown to be globally first order asymptotically optimal, with an appropriate choice of sensor thresholds \cite{tart-veer-fusion-2003,tart-veer-sqa-2008}. The procedure $\tau_\text{min}$ does not share this asymptotic optimality property.

Interestingly, while tests based on local decision making at the sensors can be shown to have the same first order performance as that of the centralized test, simulations reveal that these asymptotics ``kick in" at unreasonably low values of false alarm probability (rate). In particular, even schemes based on {\em binary} quantization at the sensors can perform better than the globally asymptotically optimum local decision based tests at moderate values of false alarm probability (rate) \cite{tart-veer-sqa-2008}. These results  point to the need for further research
on designing procedures that perform local detection at the sensors that provide good performance at moderate levels of false alarms.

\subsection{Variants of quickest change detection problem for distributed sensor systems }

In Section~\ref{sec:dss}, it is assumed for the decentralized quickest change detection problem, that the change affects all the sensors in the system simultaneously. In many practical systems it is reasonable to assume that the change will be seen by only a subset of the sensors. This problem can be modeled as quickest change detection with unknown post-change distribution, with a finite number of possibilities. A GLRT based approached can of course be used, in which multiple CuSum tests are run in parallel, corresponding to each possible post-change hypotheses. But this can be quite expensive from an implementation view point. In \cite{mei-biometrica-2010}, a CuSum based algorithm is proposed in which,
at each sensor a CuSum test is employed, the CuSum statistic is transmitted from the sensors to the fusion center, and at the fusion center, the CuSum statistics from all the sensors are added and compared with a threshold. This test has much lower computational complexity as compared to the GLR based test, and is shown to be asymptotically as good as the centralized CuSum test, as the $\FAR$ goes to zero.
Although this test is asymptotically optimal, the noise from the sensors not affected by change can degrade the performance for moderate values of false alarm rate. In \cite{mei-isit-2011}, this work of \cite{mei-biometrica-2010}, is extended to the case where information is transmitted from the sensors only when the CuSum statistic at each sensor is above a certain threshold. It is shown that this has the surprising effect of suppressing the unwanted noise and improving the performance. In \cite{xie-sieg-jstatmeet-2011}, it is proposed that a \textit{soft-thresholding} function should be used to suppress these noise terms,
and a GLRT based algorithm is proposed to detect presence of a stationary intruder (with unknown position)
in a sensor network with Gaussian observations.  A similar formulation is described in  \cite{prem-anurag-kuri-allerton-2009}.

The Bayesian decentralized quickest change detection problem under an additional constraint on the cost of observations used is studied in \cite{prem-kuma-infocom-2008}. The cost of observations is captured through the average number of observations used until the stopping time and it is shown that a threshold test similar to the Shiryaev algorithm is optimal. Recently, this problem has been studied in a minimax setting in \cite{bane-veer-ssp-2012} and asymptotically minimax algorithms have been proposed. Also, see \cite{bane-etal-ieeetwc-2011} and \cite{zach-sund-ieeetwc-2008} for other interesting energy-efficient algorithms for
quickest change detection in sensor networks.

\section{Applications of quickest change detection}
\label{sec:QCDapplications}
As mentioned in the introduction, the problem of quickest change detection has  a variety of applications.
A complete list of references to applications of quickest change detection can be quite overwhelming and therefore we only provide representative references from some of the areas. For a detailed list of references to application in  areas such as climate modeling, econometrics, environment and public health, finance, image analysis, navigation, remote sensing, etc., see \cite{poor-hadj-qcd-book-2009} and \cite{bass-niki-change-det-book-1993}.
\begin{enumerate}
 \item \textit{Statistical Process Control (SPC):} As discussed in the introduction, algorithms are required
that can detect a sudden fault arising in an industrial process or a production process. In recent years algorithms for SPC with sampling rate and sampling size control have  also been developed to minimize the cost associated with sampling \cite{taga-jqt-1998}, \cite{stou-etal-jamstaa-2000}. See \cite{stou-reyno-nla-2005}, and \cite{makis-opres-2008} and the references therein for some recent contributions.
\item \textit{Sensor networks:} As discussed in \cite{bane-veer-sqa-2012}, quickest change detection algorithms can be employed in sensor networks for infrastructure monitoring \cite{rice-etal-sss-2010}, or for habitat monitoring \cite{main-etal-wsna-2002}. Note that in these applications, the sensors are deployed for a long time, and the change may occur rarely. Therefore, data-efficient quickest change detection algorithms are needed (see Section~\ref{sec:DE-QCD}).
\item \textit{Computer network security:} Algorithms for quickest change detection have been applied in the detection of abnormal behavior in computer networks due to security breaches \cite{thot-ji-ieeetsp-2003}, \cite{tart-et-al-ieeetsp-2006}, \cite{baras-etal-ieeenet-2009}.
\item \textit{Cognitive radio:} Algorithms based on the CuSum algorithm or other quickest change detection algorithms can be developed for cooperative spectrum sensing in cognitive radio networks to detect activity of a primary user. See \cite{bane-etal-ieeetwc-2011}, \cite{llai-etal-globe-2008},  \cite{arun-vinod-ICUMT-2009} and \cite{arun-vinod-UKIWCWS-2009}.
\item \textit{Neuroscience:} The evolution of the Shiryaev algorithm is found to be similar
to the dynamics of the \textit{Leaky Integrate-and-Fire} model for neurons \cite{yu-adnips-2006}.
\item \textit{Social Networks:} It is suggested in \cite{frisen-sqa-2009} and \cite{fien-shmu-statmed-2005}
that the algorithms from the change detection literature can be employed to detect the onset of the
outbreak of a disease, or the effect of a bioterroist attack, by monitoring drug sales at retail stores.
\end{enumerate}
\section{Conclusions and future directions}
\label{sec:ConclusionFutureDirect}
In this article we reviewed the state-of-the-art in the theory of quickest change detection. We saw that while exactly or nearly optimal algorithms are available only for the i.i.d. model and for the detection of a change in a single sequence of observations, asymptotically optimal algorithms can be obtained in a much broader setting. We discussed the uniform asymptotic optimality of GLR and mixture based tests, when the post-change distribution is not known. We discussed algorithms for data-efficient quickest change detection, and showed that they are also asymptotically equivalent to their classical counterparts. For the decentralized quickest change detection model, we discussed various algorithms that are asymptotically optimal. We also reviewed the asymptotic optimality theory in the Bayesian as well as in the minimax setting for a general non-i.i.d.\ model, and showed that extensions of the Shiryaev algorithm and the CuSum algorithm to the non-i.i.d.\ setting are asymptotically optimal. Nevertheless, the list of topics discussed in this article is far from exhaustive.

Below we enumerate possible future directions in which the quickest change detection problem can be explored. We also provide references to some recent articles in which some research on these topics has been initiated.
\begin{enumerate}

\item \textit{Transient change detection:} It is assumed throughput this article that the change is persistent, i.e., once the change occurs, the system stays in the post-change state forever. In many applications
it might be more appropriate to model change as \textit{transient}, i.e., the system only stays in the post-change state for a finite duration and then returns to the pre-change state; see e.g., \cite{han-will-abra-ieeetsp-1999}, \cite{wang-will-ieeetsp-2000} and \cite{prem-etal-iwap-2010}. In this setting, in addition to false alarm and delay metrics, it may be of interest to consider metrics related to the detection of the change while the system is still in the change state.
\item \textit{Change propagation:} In applications with multiple sensors, unlike the model assumed in Section~\ref{sec:dss},  it may happen that the change does not affect all the sensors simultaneously \cite{hadj-zhan-poor-ieeetit-2009}. The change may \textit{propagate} from one sensor to the next, with the statistics of the propagation process being known before hand \cite{ragh-veer-ieeetit-2010}.
\item \textit{Multiple change points in networks}:  In some monitoring application there may be multiple change points that affect different sensors in a network, and the goal is to exploit the relationship between the change points and the sensors affected to detect the changes \cite{nguy-amin-raja-isit-2012}.
 \item \textit{Quickest change detection with social learning:} In the classical quickest change detection problem, to compute the statistic at each time step, the decision maker has access to the entire past observations. An interesting variant of the problem is quickest change detection with social learning, when the time index is replaced by an agent index, i.e., when the statistic is updated over agents and not over time, and the agents do not have access to the entire past history of observations but only to some processed version (e.g., binary decisions) from the previous agent; see \cite{kris-icassp-2012}, \cite{kris-ieeetit-2011} and \cite{kris-ieeetit-2012}.
\item \textit{Change Detection with Simultaneous Classification:}
In many applications, the post-change distribution is not uniquely specified and may come from one of multiple hypotheses $H_1,\dots,H_M$, in which case along with detecting the change it of interest to identify which hypothesis is true. See, e.g., \cite{niki-ieeetit-2003} and \cite{tart-sqa-2008}.
\item \textit{Synchronization issues:} If a quickest change detection algorithm is implemented in a sensor network where sensors communicate with the fusion center using a MAC protocol, the fusion center might receive asynchronous information from the sensors due to  networking delays. It is of interest to develop algorithms that can detect changes while handling MAC layer issues; see, e.g.,  \cite{prem-anurag-kuri-allerton-2009}.
\end{enumerate}

\section*{Acknowledgement}
The authors would like to thank Prof. Abdelhak Zoubir for useful suggestions, and Mr. Michael Fauss  and Mr. Shang Kee Ting for their detailed reviews of an early draft of this work. The authors would also like to acknowledge the support of  the National Science Foundation under grants CCF 0830169 and CCF 1111342, through the University of Illinois at Urbana-Champaign, and  the U.S. Defense Threat Reduction Agency through subcontract 147755 at the University of Illinois from prime award HDTRA1-10-1-0086.

\pagebreak
\bibliographystyle{ieeetr}


\bibliography{QCD_verSubmitted}

\begin{thebibliography}{10}

\bibitem{shew-jamstaa-1925}
W.~A. Shewhart, ``The application of statistics as an aid in maintaining
  quality of a manufactured product,'' {\em J. Amer. Statist. Assoc.}, vol.~20,
  pp.~546--548, Dec. 1925.

\bibitem{shew-book-1931}
W.~A. Shewhart, {\em Economic control of quality of manufactured product}.
\newblock American Society for Quality Control, 1931.

\bibitem{page-biometrica-1954}
E.~S. Page, ``Continuous inspection schemes,'' {\em Biometrika}, vol.~41,
  pp.~100--115, June 1954.

\bibitem{shir-siamtpa-1963}
A.~N. Shiryaev, ``On optimum methods in quickest detection problems,'' {\em
  Theory of Prob and App.}, vol.~8, pp.~22--46, 1963.

\bibitem{shir-opt-stop-book-1978}
A.~N. Shiryayev, {\em Optimal Stopping Rules}.
\newblock New York: Springer-Verlag, 1978.

\bibitem{lord-amstat-1971}
G.~Lorden, ``Procedures for reacting to a change in distribution,'' {\em Ann.
  Math. Statist.}, vol.~42, pp.~1897--1908, Dec. 1971.

\bibitem{mous-astat-1986}
G.~V. Moustakides, ``Optimal stopping times for detecting changes in
  distributions,'' {\em Ann. Statist.}, vol.~14, pp.~1379--1387, Dec. 1986.

\bibitem{ritov-astat-1990}
Y.~Ritov, ``Decision theoretic optimality of the {CUSUM} procedure,'' {\em Ann.
  Statist.}, vol.~18, pp.~1464--1469, Nov. 1990.

\bibitem{poll-astat-1985}
M.~Pollak, ``Optimal detection of a change in distribution,'' {\em Ann.
  Statist.}, vol.~13, pp.~206--227, Mar. 1985.

\bibitem{lai-ieeetit-1998}
T.~L. Lai, ``Information bounds and quick detection of parameter changes in
  stochastic systems,'' {\em IEEE Trans. Inf. Theory}, vol.~44, pp.~2917
  --2929, Nov. 1998.

\bibitem{polu-tart-mcap-2011}
A.~Polunchenko and A.~G. Tartakovsky, ``State-of-the-art in sequential
  change-point detection,'' {\em Method. and Comp. in App. Probab.}, pp.~1--36,
  Oct. 2011.

\bibitem{tart-veer-siamtpa-2005}
A.~G. Tartakovsky and V.~V. Veeravalli, ``General asymptotic {Bayesian} theory
  of quickest change detection,'' {\em SIAM Theory of Prob. and App.}, vol.~49,
  pp.~458--497, Sept. 2005.

\bibitem{poor-hadj-qcd-book-2009}
H.~V. Poor and O.~Hadjiliadis, {\em Quickest detection}.
\newblock Cambridge University Press, 2009.

\bibitem{chow-robb-sieg-book-1971}
Y.~S. Chow, H.~Robbins, and D.~Siegmund, {\em Great expectations: the theory of
  optimal stopping}.
\newblock Houghton Mifflin, 1971.

\bibitem{tart-niki-bass-2013}
A.~G. Tartakovsky, I.~V. Nikiforov, and M.~Basseville, {\em Sequential
  Analysis: {Hypothesis} Testing and Change-Point Detection}.
\newblock Statistics, CRC Press, 2013.

\bibitem{will-book-probmart-1991}
D.~Williams, {\em Probability With Martingales}.
\newblock Cambridge Mathematical Textbooks, Cambridge University Press, 1991.

\bibitem{sieg-seq-anal-book-1985}
D.~Siegmund, {\em Sequential Analysis: {Tests} and Confidence Intervals}.
\newblock Springer series in statistics, Springer-Verlag, 1985.

\bibitem{wood-nonlin-ren-th-book-1982}
M.~Woodroofe, {\em Nonlinear Renewal Theory in Sequential Analysis}.
\newblock CBMS-NSF regional conference series in applied mathematics, SIAM,
  1982.

\bibitem{bert-dyn-prog-book-1995}
D.~Bertsekas, {\em Dynamic Programming and Optimal Control, Vol. I and II}.
\newblock Belmont, Massachusetts: Athena Scientific, 1995.

\bibitem{tart-poll-polu-arxiv-2011}
A.~G. Tartakovsky, M.~Pollak, and A.~Polunchenko, ``Third-order asymptotic
  optimality of the generalized {Shiryaev-Roberts} changepoint detection
  procedures,'' {\em ArXiv e-prints}, May 2010.

\bibitem{robe-technometrics-1966}
S.~W. Roberts, ``A comparison of some control chart procedures,'' {\em
  Technometrics}, vol.~8, pp.~411--430, Aug. 1966.

\bibitem{poll-tart-statsinica-2009}
M.~Pollak and A.~G. Tartakovsky, ``Optimality properties of the
  {Shiryaev-Roberts} procedure,'' {\em Statistica Sinica}, vol.~19,
  pp.~1729--1739, 2009.

\bibitem{mous-polu-tart-statsin}
G.~V. Moustakides, A.~S. Polunchenko, and A.~G. Tartakovsky, ``A numerical
  approach to performance analysis of quickest change-point detection
  procedures,'' {\em Statistica Sinica}, vol.~21, pp.~571--596, 2011.

\bibitem{poll-astat-1987}
M.~Pollak, ``Average run lengths of an optimal method of detecting a change in
  distribution,'' {\em Ann. Statist.}, vol.~15, pp.~749--779, June 1987.

\bibitem{wald-wolf-amstat-1948}
A.~Wald and J.~Wolfowitz, ``Optimum character of the sequential probability
  ratio test,'' {\em Ann. Math. Statist.}, vol.~19, no.~3, pp.~pp. 326--339,
  1948.

\bibitem{mous-astat-2008}
G.~V. Moustakides, ``Sequential change detection revisited,'' {\em Ann.
  Statist.}, vol.~36, pp.~787--807, Apr. 2008.

\bibitem{poll-sieg-atat-1975}
M.~Pollak and D.~Siegmund, ``Approximations to the expected sample size of
  certain sequential tests,'' {\em Ann. Statist.}, vol.~3, pp.~1267--1282, Nov.
  1975.

\bibitem{lai-jrss-1995}
T.~L. Lai, ``Sequential changepoint detection in quality control and dynamical
  systems,'' {\em J. Roy. Statist. Soc. Suppl.}, vol.~57, no.~4, pp.~pp.
  613--658, 1995.

\bibitem{sieg-venk-astat-1995}
D.~Siegmund and E.~S. Venkatraman, ``Using the generalized likelihood ratio
  statistic for sequential detection of a change-point,'' {\em Ann. Statist.},
  vol.~23, pp.~255--271, Feb. 1995.

\bibitem{lai-xing-sqa-2010}
T.~L. Lai and H.~Xing, ``Sequential change-point detection when the pre- and
  post-change parameters are unknown,'' {\em Sequential Analysis}, vol.~29,
  pp.~162--175, May 2010.

\bibitem{unni-etal-ieeeit-2011}
J.~Unnikrishnan, V.~V. Veeravalli, and S.~P. Meyn, ``Minimax robust quickest
  change detection,'' {\em IEEE Trans. Inf. Theory}, vol.~57, pp.~1604 --1614,
  Mar. 2011.

\bibitem{bane-veer-sqa-2012}
T.~Banerjee and V.~V. Veeravalli, ``Data-efficient quickest change detection
  with on-off observation control,'' {\em Sequential Analysis}, vol.~31,
  pp.~40--77, Feb. 2012.

\bibitem{bane-veer-icassp-2012}
T.~Banerjee and V.~V. Veeravalli, ``Data-efficient minimax quickest change
  detection,'' in {\em IEEE Conference on Acoustics, Speech, and Signal
  Processing (ICASSP)}, pp.~3937--3940, Mar. 2012.

\bibitem{tsit-dec-det-1993}
J.~N. Tsitsiklis, ``Decentralized detection,'' in {\em Advances in Statistical
  Signal Processing} (H.~V. Poor and J.~B. Thomas, eds.), vol.~2, Greenwich,
  CT: JAI Press, 1993.

\bibitem{vars-book-1996}
P.~K. Varshney, {\em Distributed detection and data fusion}.
\newblock New York: Springer Verlag, 1996.

\bibitem{will-swas-blum-ieeetsp-2000}
P.~Willett, P.~F. Swaszek, and R.~S. Blum, ``The good, bad and ugly:
  {D}istributed detection of a known signal in dependent {G}aussian noise,''
  {\em IEEE Trans. Signal Process.}, vol.~48, pp.~3266--3279, Dec. 2000.

\bibitem{cham-veer-ieeetsp-2003}
J.~F. Chamberland and V.~V. Veeravalli, ``Decentralized detection in sensor
  networks,'' {\em IEEE Trans. Signal Process.}, vol.~51, pp.~407--416, Feb.
  2003.

\bibitem{cham-veer-ieeespm-2007}
J.-F. Chamberland and V.~V. Veeravalli, ``Wireless sensors in distributed
  detection applications,'' {\em IEEE Signal Processing Magazine Special Issue
  on Resource-Constrained Signal Processing, Communications, and Networking},
  vol.~24, pp.~16--25, May 2007.

\bibitem{veer-vars-survey-2012}
V.~V. Veeravalli and P.~K. Varshney, ``Distributed inference in wireless sensor
  networks,'' {\em Philosophical Transactions of the Royal Society A:
  Mathematical, Physical and Engineering Sciences}, vol.~370, pp.~100--117,
  Jan. 2012.

\bibitem{ho-procieee-1980}
Y.~Ho, ``Team decision theory and information structures,'' {\em Proc. IEEE},
  vol.~68, pp.~644 -- 654, June 1980.

\bibitem{veer-ieeetit-2001}
V.~V. Veeravalli, ``Decentralized quickest change detection,'' {\em IEEE Trans.
  Inf. Theory}, vol.~47, pp.~1657--1665, May 2001.

\bibitem{tart-veer-asm-book-2004}
A.~G. Tartakovsky and V.~V. Veeravalli, ``Change-point detection in
  multichannel and distributed systems,'' in {\em Applied Sequential
  Methodologies: Real-World Examples with Data Analysis} (N.~Mukhopadhyay,
  S.~Datta, and S.~Chattopadhyay, eds.), vol.~173 of {\em Statistics: a Series
  of Textbooks and Monographs}, pp.~339--370, New York, USA: Marcel Dekker,
  Inc, 2004.

\bibitem{tart-veer-sqa-2008}
A.~G. Tartakovsky and V.~V. Veeravalli, ``Asymptotically optimal quickest
  change detection in distributed sensor systems,'' {\em Sequential Analysis},
  vol.~27, pp.~441--475, Oct. 2008.

\bibitem{mei-ieeetit-2005}
Y.~Mei, ``Information bounds and quickest change detection in decentralized
  decision systems,'' {\em IEEE Trans. Inf. Theory}, vol.~51, pp.~2669 --2681,
  July 2005.

\bibitem{tart-kim-fusion-2006}
A.~G. Tartakovsky and H.~Kim, ``Performance of certain decentralized
  distributed change detection procedures,'' in {\em IEEE International
  Conference on Information Fusion}, (Florence, Italy), pp.~1--8, July 2006.

\bibitem{tart-veer-fusion-2003}
A.~G. Tartakovsky and V.~V. Veeravalli, ``Quickest change detection in
  distributed sensor systems,'' in {\em IEEE International Conference on
  Information Fusion}, (Cairns, Australia), pp.~756--763, July 2003.

\bibitem{mei-biometrica-2010}
Y.~Mei, ``Efficient scalable schemes for monitoring a large number of data
  streams,'' {\em Biometrika}, vol.~97, pp.~419--433, Apr. 2010.

\bibitem{mei-isit-2011}
Y.~Mei, ``Quickest detection in censoring sensor networks,'' in {\em IEEE
  International Symposium on Information Theory (ISIT)}, pp.~2148 --2152, Aug.
  2011.

\bibitem{xie-sieg-jstatmeet-2011}
Y.~Xie and D.~Siegmund, ``Multi-sensor change-point detection,'' in {\em Joint
  Statistical Meetings}, Aug. 2011.

\bibitem{prem-anurag-kuri-allerton-2009}
K.~Premkumar, A.~Kumar, and J.~Kuri, ``Distributed detection and localization
  of events in large ad hoc wireless sensor networks,'' in {\em Allerton
  Conference on Communication, Control, and Computing}, pp.~178 --185, Oct.
  2009.

\bibitem{prem-kuma-infocom-2008}
K.~Premkumar and A.~Kumar, ``Optimal sleep-wake scheduling for quickest
  intrusion detection using wireless sensor networks,'' in {\em IEEE Conference
  on Computer Communications (INFOCOM)}, pp.~1400--1408, Apr. 2008.

\bibitem{bane-veer-ssp-2012}
T.~Banerjee and V.~V. Veeravalli, ``Energy-efficient quickest change detection
  in sensor networks,'' in {\em IEEE Statistical Signal Processing Workshop},
  Aug. 2012.

\bibitem{bane-etal-ieeetwc-2011}
T.~Banerjee, V.~Sharma, V.~Kavitha, and A.~K. JayaPrakasam, ``Generalized
  analysis of a distributed energy efficient algorithm for change detection,''
  {\em IEEE Trans. Wireless Commun.}, vol.~10, pp.~91--101, Jan. 2011.

\bibitem{zach-sund-ieeetwc-2008}
L.~Zacharias and R.~Sundaresan, ``Decentralized sequential change detection
  using physical layer fusion,'' {\em IEEE Trans. Wireless Commun.}, vol.~7,
  pp.~4999--5008, Dec. 2008.

\bibitem{bass-niki-change-det-book-1993}
M.~Basseville and I.~V. Nikiforov, {\em Detection of Abrupt Changes: {Theory}
  and Application}.
\newblock Englewood Cliffs, NJ: Prentice Hall, 1993.

\bibitem{taga-jqt-1998}
G.~Tagaras, ``A survey of recent developments in the design of adaptive control
  charts,'' {\em Journal of Quality Technology}, vol.~30, pp.~212--231, July
  1998.

\bibitem{stou-etal-jamstaa-2000}
Z.~G. Stoumbos, M.~R. Reynolds, T.~P. Ryan, and W.~H. Woodall, ``The state of
  statistical process control as we proceed into the 21st century,'' {\em J.
  Amer. Statist. Assoc.}, vol.~95, pp.~992--998, Sept. 2000.

\bibitem{stou-reyno-nla-2005}
Z.~G. Stoumbos and M.~R. Reynolds, ``Economic statistical design of adaptive
  control schemes for monitoring the mean and variance: An application to
  analyzers,'' {\em Nonlinear Analysis: Real World Applications}, vol.~6,
  pp.~817 -- 844, Dec. 2005.

\bibitem{makis-opres-2008}
V.~Makis, ``Multivariate bayesian control chart,'' {\em Operations Research},
  vol.~56, pp.~487--496, Mar. 2008.

\bibitem{rice-etal-sss-2010}
J.~A. Rice, K.~Mechitov, S.~Sim, T.~Nagayama, S.~Jang, R.~Kim, B.~F. Spencer,
  G.~Agha, and Y.~Fujino, ``Flexible smart sensor framework for autonomous
  structural health monitoring,'' {\em Smart Structures and Systems}, vol.~6,
  no.~5-6, pp.~423--438, 2010.

\bibitem{main-etal-wsna-2002}
A.~Mainwaring, D.~Culler, J.~Polastre, R.~Szewczyk, and J.~Anderson, ``Wireless
  sensor networks for habitat monitoring,'' in {\em Proceedings of the 1st ACM
  international workshop on Wireless sensor networks and applications}, WSNA
  '02, (New York, NY, USA), pp.~88--97, ACM, Sept. 2002.

\bibitem{thot-ji-ieeetsp-2003}
M.~Thottan and C.~Ji, ``Anomaly detection in {IP} networks,'' {\em IEEE Trans.
  Signal Process.}, vol.~51, pp.~2191--2204, Aug. 2003.

\bibitem{tart-et-al-ieeetsp-2006}
A.~G. Tartakovsky, B.~L. Rozovskii, R.~B. Blazek, and H.~Kim, ``A novel
  approach to detection of intrusions in computer networks via adaptive
  sequential and batch-sequential change-point detection methods,'' {\em IEEE
  Trans. Signal Process.}, vol.~54, pp.~3372--3382, Sept. 2006.

\bibitem{baras-etal-ieeenet-2009}
A.~A. Cardenas, S.~Radosavac, and J.~S. Baras, ``Evaluation of detection
  algorithms for {MAC} layer misbehavior: {T}heory and experiments,'' {\em
  IEEE/ACM Trans. Netw.}, vol.~17, pp.~605 --617, Apr. 2009.

\bibitem{llai-etal-globe-2008}
L.~Lai, Y.~Fan, and H.~V. Poor, ``Quickest detection in cognitive radio: {A}
  sequential change detection framework,'' in {\em IEEE GLOBECOM}, pp.~1 --5,
  Dec. 2008.

\bibitem{arun-vinod-ICUMT-2009}
A.~K. Jayaprakasam and V.~Sharma, ``Cooperative robust sequential detection
  algorithms for spectrum sensing in cognitive radio,'' in {\em International
  Conference on Ultramodern Telecommunications (ICUMT)}, pp.~1 --8, Oct. 2009.

\bibitem{arun-vinod-UKIWCWS-2009}
A.~K. Jayaprakasam and V.~Sharma, ``Sequential detection based cooperative
  spectrum sensing algorithms in cognitive radio,'' in {\em First UK-India
  International Workshop on Cognitive Wireless Systems (UKIWCWS)}, pp.~1 --6,
  Dec. 2009.

\bibitem{yu-adnips-2006}
A.~J. Yu, ``Optimal change-detection and spiking neurons,'' in {\em Advances in
  Neural Information Processing Systems 19} (B.~Sch\"{o}lkopf, J.~Platt, and
  T.~Hoffman, eds.), pp.~1545--1552, Cambridge, MA: MIT Press, 2007.

\bibitem{frisen-sqa-2009}
M.~Frisen, ``Optimal sequential surveillance for finance, public health, and
  other areas,'' {\em Sequential Analysis}, vol.~28, pp.~310--337, July 2009.

\bibitem{fien-shmu-statmed-2005}
S.~E. Fienberg and G.~Shmueli, ``Statistical issues and challenges associated
  with rapid detection of bio-terrorist attacks.,'' {\em Statistics in
  Medicine}, vol.~24, pp.~513--529, Feb. 2005.

\bibitem{han-will-abra-ieeetsp-1999}
C.~Han, P.~K. Willett, and D.~A. Abraham, ``Some methods to evaluate the
  performance of {P}age's test as used to detect transient signals,'' {\em IEEE
  Trans. Signal Process.}, vol.~47, pp.~2112--2127, Aug. 1999.

\bibitem{wang-will-ieeetsp-2000}
Z.~Wang and P.~K. Willett, ``A performance study of some transient detectors,''
  {\em IEEE Trans. Signal Process.}, vol.~48, pp.~2682--2685, Sept. 2000.

\bibitem{prem-etal-iwap-2010}
K.~Premkumar, A.~Kumar, and V.~V. Veeravalli, ``Bayesian quickest transient
  change detection,'' in {\em International Workshop on Applied Probability
  (IWAP)}, July 2010.

\bibitem{hadj-zhan-poor-ieeetit-2009}
O.~Hadjiliadis, H.~Zhang, and H.~V. Poor, ``One shot schemes for decentralized
  quickest change detection,'' {\em IEEE Trans. Inf. Theory}, vol.~55,
  pp.~3346--3359, July 2009.

\bibitem{ragh-veer-ieeetit-2010}
V.~Raghavan and V.~V. Veeravalli, ``Quickest change detection of a {M}arkov
  process across a sensor array,'' {\em IEEE Trans. Inf. Theory}, vol.~56,
  pp.~1961--1981, Apr. 2010.

\bibitem{nguy-amin-raja-isit-2012}
X.~Nguyen, A.~Amini, and R.~Rajagopal, ``Message-passing sequential detection
  of multiple change points in networks,'' in {\em IEEE International Symposium
  on Information Theory (ISIT)}, July 2012.

\bibitem{kris-icassp-2012}
V.~Krishnamurthy, ``Quickest time change detection with social learning,'' in
  {\em IEEE International Conference on Acoustics, Speech, and Signal
  Processing (ICASSP)}, pp.~5257--5260, Mar. 2012.

\bibitem{kris-ieeetit-2011}
V.~Krishnamurthy, ``Bayesian sequential detection with phase-distributed change
  time and nonlinear penalty; a {POMDP} lattice programming approach,'' {\em
  IEEE Trans. Inf. Theory}, vol.~57, pp.~7096 --7124, Oct. 2011.

\bibitem{kris-ieeetit-2012}
V.~Krishnamurthy, ``Quickest detection with social learning: Interaction of
  local and global decision makers,'' {\em ArXiv e-prints}, Mar. 2012.

\bibitem{niki-ieeetit-2003}
I.~V. Nikiforov, ``A lower bound for the detection/isolation delay in a class
  of sequential tests,'' {\em IEEE Trans. Inf. Theory}, vol.~49,
  pp.~3037--3046, Nov. 2003.

\bibitem{tart-sqa-2008}
A.~G. Tartakovsky, ``Multidecision quickest change-point detection: Previous
  achievements and open problems,'' {\em Sequential Analysis}, vol.~27,
  pp.~201--231, Apr. 2008.

\end{thebibliography}


Venugopal V. Veeravalli (M'92--SM'98--F'06) received the B.Tech. degree (Silver Medal Honors) from the Indian Institute of Technology, Bombay, in 1985, the M.S. degree from Carnegie Mellon University, Pittsburgh, PA, in 1987, and the Ph.D. degree from the University of Illinois at Urbana-Champaign, in 1992, all in electrical engineering.

He joined the University of Illinois at Urbana-Champaign in 2000, where he is currently a Professor in the department of Electrical and Computer Engineering and the Coordinated Science Laboratory. He served as a Program Director for communications research at the U.S. National Science Foundation in Arlington, VA from 2003 to 2005.  He has previously held academic positions at Harvard University, Rice University, and Cornell University, and has been on sabbatical at MIT, IISc Bangalore, and Qualcomm, Inc. His research interests include wireless communications, distributed sensor systems and networks, detection and estimation theory, and information theory.

Prof. Veeravalli was a Distinguished Lecturer for the IEEE Signal Processing Society during 2010--2011. He has been on the Board of Governors of the IEEE Information Theory Society. He has been an Associate Editor for Detection and Estimation for the IEEE Transactions on Information Theory and for the IEEE Transactions on Wireless Communications. Among the awards he has received for research and teaching are the IEEE Browder J. Thompson Best Paper Award, the National Science Foundation CAREER Award, and the Presidential Early Career Award for Scientists and Engineers (PECASE).

{\bf Taposh Banerjee:} Taposh Banerjee received an M.E. in Telecommunications from the ECE Department of the Indian Institute of Science. He is now pursuing his Ph.D. at the Coordinated Science Laboratory and the ECE Department at the University of Illinois at Urbana-Champaign. His research interests are in detection and estimation theory, sequential analysis and wireless communications and networks.

\end{document}